\documentclass[12pt]{amsart}
\usepackage{geometry} %调整页边距
\usepackage{amsmath}%用于编写公式
\usepackage{amsthm}
\usepackage{bbm} %可以使用小写空心体
\usepackage{amsfonts,amssymb}
\usepackage{fancyhdr} %用于调整页眉页脚
\usepackage{amsfonts,amssymb}  %可以使用空心体
\usepackage{mathrsfs}  %可以使用花体字母
\usepackage{graphicx}
\usepackage[all]{xy}
\usepackage{hyperref}
\usepackage{times}

\usepackage{float} 
\usepackage{amsthm}
\renewcommand\proofname{Proof}

\newtheorem{thm}{Theorem}[section]
\newtheorem{lem}[thm]{Lemma}
\newtheorem{cor}[thm]{Corollary}
\newtheorem{prop}[thm]{Proposition}

\newtheorem{rmk}[thm]{Remark}

\newtheorem{defi}[thm]{Definition}

\def\log{{\rm log}}

\def\ord{{\rm ord}}
\def\min{{\rm min}}
\def\max{{\rm max}}

\def\sup{{\rm sup}}
\def\lim{{\rm lim}}
\def\limsup{{\rm lim\,sup}}

\def\gr{{\rm gr}}

\def\Rees{{\rm Rees}}
\def\Fut{{\rm Fut}}
\def\Proj{{\rm Proj}}

\def\log{{\rm log}}

\def\reg{{\rm reg}}

\def\sing{{\rm sing}}
\def\red{{\rm red}}

\newcommand{\iddbar}{{\sqrt{-1}\partial \bar{\partial}}}

\newcommand{\IC}{{\mathbb C}}

\newcommand{\IN}{{\mathbb N}}
 
\newcommand{\IP}{{\mathbb P}} 
\newcommand{\IQ}{{\mathbb Q}} 
\newcommand{\IR}{{\mathbb R}}

\newcommand{\IZ}{{\mathbb Z}}

\newcommand{\CE}{{\mathcal E}}
\newcommand{\CF}{{\mathcal F}}

\newcommand{\CK}{{\mathcal K}}
\newcommand{\CL}{{\mathcal L}}

\newcommand{\CO}{{\mathcal O}}

\newcommand{\CU}{{\mathcal U}}

\newcommand{\CX}{{\mathcal X}}

\newcommand{\Ric}{\mathrm{Ric}}

\newcommand{\seq}{\subseteq}
\newcommand{\lam}{\lambda}

\newcommand{\la}{\langle}
\newcommand{\ra}{\rangle}

\newcommand{\D}{\Delta}

\title{Bergman kernels on degenerations}

\author{Linsheng Wang}
\address{Department of Mathematics\\Nanjing University\\
Nanjing\\ 210008\\ China}
\email{linsheng\_wang@outlook.com}

\author{Shengxuan Zhou}
\address{Beijing International Center for Mathematical Research\\
Peking University\\
Beijing\\ 100871\\ China}
\email{zhoushx19@pku.edu.cn}

\thanks{}
\date{}
\dedicatory{}

\begin{document}

\begin{abstract}
We introduce the fiberwise Bergman kernel for a flat family of polarized varieties over a Riemann surface, which extends the classical Bergman kernel defined on the reduced fibers. We establish the continuity of the fiberwise Bergman kernel and provide a result on uniform convergence for the Fubini-Study currents. As a consequence, we show that the fiberwise Bergman kernel on test configurations exhibits continuity, and the Fubini-Study currents converge uniformly. 
\end{abstract}
\maketitle
\tableofcontents

\section{Introduction}

The Bergman kernels of polarized manifolds serve a significant purpose in K\"ahler geometry as a localized version of the Hilbert polynomial. In the seminal work \cite{tia90a}, Tian utilized his peak section technique to establish that the Bergman metrics corresponding to $L^m$ approach the original polarized metric as $m$ tends to infinity in the $C^2$-topology. Moreover, the Bergman kernels are the K\"ahler potentials of the Bergman metrics. Subsequent to Tian's work, numerous studies have investigated the asymptotic behavior of Bergman kernels and Bergman metrics on a single polarized manifold, as documented in \cite{rua98, zel98, cat97, ll16, dlm06}. This collection of research has yielded a localized version of the asymptotic Hirzebruch-Riemann-Roch formula, recognized as the Tian-Yau-Zelditch expansion. In the case of holomorphic submersions, Ma-Zhang \cite{mz22} have formulated a local version of the asymptotic Riemann-Roch-Grothendieck theorem.

Moreover, an intriguing inquiry is the characterization of Bergman kernels on a particular family of polarized K\"ahler manifolds. One prominent problem in the study of Bergman kernels, known as the partial $C^0$-estimate, is to determine whether there exists a uniform positive lower bound for Bergman kernels on a specific class of polarized K\"ahler manifolds. The partial $C^0$-estimate was first introduced in Tian's work \cite{tia90b} on the search for K\"ahler-Einstein metrics on del Pezzo surfaces and is often regarded as an effective version of the finite generation of the ring $R(X,L) =\oplus_{k\geq 0} H^0 \left( X,L^k \right)$ (see for example \cite{chili1}), akin to Matsusaka's big theorem as an effective version of Kodaira's embedding theorem. The partial $C^0$-estimate, assuming a lower bound on Ricci curvature and non-collapsing condition, has been a powerful tool in the study of K-stability and the Yau-Tian-Donaldson conjecture \cite{tia13,tia15,cds15}, and there have been many works dedicated to this problem \cite{ds14,jia16,ls22,sze16,wz21,zha21a}. It is worth noting that the partial $C^0$-estimate is closely related to the convergence of the Bergman kernel.

This article explores the Bergman kernels on a flat family of complex spaces, specifically defining fiberwise Bergman kernels for flat families over Riemann surfaces and analyzing the continuity of these kernels on the total spaces. When all fibers of the family are reduced complex spaces, the fiberwise Bergman kernel coincides with the standard Bergman kernels on fibers. However, in many cases (including some of interest), the fibers of a flat family are not all reduced. By studying the behavior of the fiberwise Bergman kernel on the total space of degeneration, we can gain insights into the geometric properties of flat families of complex spaces. 

Let us define the context of our study. Consider a normal complex space $\CX$ and a Riemann surface $C$. Let $\pi:\CX\to C$ be a proper surjective morphism with reduced general fibers. Suppose that $\CL$ is a line bundle on $\CX$ equipped with a continuous Hermitian metric $h$, and $\omega$ is a continuous $2$-form on $\CX$ such that its restriction to each fiber is a Hermitian form. Denote the scheme-theoretical fiber of $\pi$ at $t\in C$ by $\CX_t=\pi^{-1}(t)$, which may not necessarily be reduced. The first theorem of our paper is as follows. 

\begin{thm}
\label{maintheorem}
The following statements are equivalent:
\begin{enumerate}
\item The function $h^0(\CX_t, \CL_t)=\dim H^0(\CX_t, \CL_t)$ is constant for all $t\in C$.
\item The fiberwise Bergman kernel (as defined in Definition \ref{definitionfiberwise bergmankernel}) $\rho_{\CL_{\pi (x)} } (x) $ is continuous on $\CX$.
\end{enumerate}
\end{thm}

\begin{rmk} \rm
The theorem remains valid even when $C$ is a higher-dimensional normal complex space, provided that we assume $\pi$ to be flat with reduced fibers. It is worth noting that when $C$ is a Riemann surface, the flatness of $\pi$ follows automatically. For further information, see Lemma \ref{lemmariemannsurfaceflat}.
\end{rmk}

The Bergman kernel can be interpreted as a local version of $0$-th cohomology. Therefore, the continuity of the fiberwise Bergman kernel on fiber spaces can be seen as a local version of the $0$-th Betti numbers of the line bundle's restriction on the fibers being constant. 

Next result of this paper goes back to Tian's work \cite{tia90a} which says that any polarized metric can be approximated by the Bergman metrics. This result has been extended to all Fubini-Study currents (may not be rational) by \cite{cmm17}. With the notation of fiberwise Bergman kernels, we can generlize the approximation result into a uniform version. 

Consider Hermitian line bundles $\{F_j\}_{1\le j\le k}$ on $\CX$ equipped with continuous Hermitian metrics $h_j$. We denote by $h_{j,t}$ the restriction of $h_j$ on $\CX_t$. Assume that $c_1(h_{1,t})\ge \varepsilon\omega_t$ for some uniform constant $\varepsilon>0$ with respect to all $t\in C$, and $c_1(h_{j,t})\ge 0 $ for all $t\in C$ and $2\le j\le k$. For any $1\le j \le k$, let $\{p_{j,m}\}$ be a sequence of positive integers with $\lim_{m\to\infty} \frac{p_{j,m}}{m}=r_j > 0$. Denote $\CL_m = F_1^{p_{1,m}} \otimes \cdots \otimes F_k^{p_{k,m}}$ and $h^{\CL_m}=\prod_{j=1}^k h_{j}^{p_{j,m}}$. The fiberwise Bergman kernel function of $(\CL_m, h^{\CL_m})$ on $\CX$ is simply denoted as $\rho_m(x)=\rho_{\CX_t, \CL_{m,t}}(x)$ for $x\in \CX$ and $t=\pi(x)$. Note that $\sum_{j} r_j c_1 (h_{j,t}) $ is not a integral class in general. 

\begin{thm} \label{uniformly convergencechap1}
The series of functions 
$$\varphi_m(t):=\int_{\CX_t}\frac{1}{m}
|\log(\rho_m)|\omega_t^n$$
converges to $0$ uniformly on compact subsets of $C$ as $m\to\infty$. Consequently, for any smooth $(n-1,n-1)$-form $\alpha$ on $\CX$, the sequence
$$
\int_{\CX_t}\left( \sum_{j=1}^k r_j c_1(h_{j,t}) - 
\frac{1}{m}\omega_{FS,m,t} \right)
\wedge \alpha$$
converges to $0$ uniformly on compact subsets of $C$ as $m\to \infty$.
\end{thm}

There are some examples that Theorem \ref{maintheorem} and Theorem \ref{uniformly convergencechap1} will apply.

Let $\CL$ be a line bundle on $\CX$ which is relatively ample with respect to $\pi$ (or simply $\pi$-ample), that is, there exists $r\in \IN$, and holomorphic sections $s_0,\cdots,s_N \in H^0(\CX, r\CL)$ inducing a closed embedding $\CX \to \IP^N\times C$, which is compatible with $\pi$. The line bundle $r\CL$ is called relatively very ample with respect to $\pi$. In this case, we conclude that for sufficiently large integer $m>0 $, the fiberwise Bergman kernel of $\CL^m $ is continuous. 

The first nontrivial example is the algebraic elliptic surfaces. Let $(\CX, \CL)$ be a polarized (hence algebraic) elliptic surface, with the corresponding elliptic fibration $\pi:\CX \to C$. By Kodaira's work \cite{kod60, kod63} on the classification of elliptic surfaces, the singular fibers of $\pi$ can be classified into $8$ different types, and most of them are non-reduced, see \cite[Theorem 6.2]{kod63}. 
By the above Corollary, we see that for sufficiently large $m\in \IN$, the fibrewise Bergman kernel of $\CL^m $ is a continuous function on $\CX$. 

%For higher dimensional fibrations, one may consider the semisimple degenerations of polarized varieties. 

Another example is the test configuration (or special degeneration) introduced by \cite{tia97, don02} in studying the existence of K\"ahler-Einstein metric on Fano manifolds. The so-called Yau-Tian-Donaldson conjecture predicts that a Fano manifold $X$ with no nontrivial holomorphic vector field admits K\"ahler-Einstein metric if and only if it is K-stable, that is, for any ample test configuration $(\CX,\CL)$ of $(X, -K_X)$, the generalized Futaki invariant $\Fut(\CX,\CL)$ is positive. A test configuration is a polarized family $\pi:(\CX,\CL)\to \IC$ with a $\IC^*$-action making $\pi$ equivariant. Hence the general fibers are isomorphic, but the central fiber $\CX_0$ may be reducible and non-reduced. See Section \ref{tc} for details.

\begin{cor}
\label{maincoro2}
For any (semiample) test configuration $(\CX,\CL)$ of a polarized projective variety $(X,L)$, the fiberwise Bergman kernels $\rho_m(x)$ of $(\CX, r\CL, h, \omega)$ are continuous on $\CX$, where $r\in \IN$ and $r\CL$ is a globally generated line bundle. Moreover, $ \int_{\CX_t}\frac{1}{m}
|\log(\rho_m)|\omega_t^n$ converges to $0$ uniformly on compact subsets of $C$ as $m\to\infty$.
\end{cor}

\begin{rmk} \rm
In general, the central fiber $\CX_0$ of a test configuration is not a normal space. As far as the authors are aware, this corollary is novel even when only considering the convergence of the Fubini-Study currents on the central fiber.
\end{rmk}

This paper is organized as follows. In Section \ref{section_Fiberwise Bergman kernel}, we recall the classical Bergman kernel and give the definition of the fiberwise Bergman kernel. We prove Theorem \ref{maintheorem} in Section \ref{section_continuity}, and Corollary \ref{maincoro2} in Section \ref{section_uniform convergence}. Some applications of the above theorems will be discussed in Section \ref{section_Applications}.

\section{Fiberwise Bergman kernels}
\label{section_Fiberwise Bergman kernel}
In this section, we define the fiberwise Bergman kernels for proper flat morphisms, and we show that the continuity of the fiberwise Bergman kernels implies the constancy of $0^{th}$ betti numbers of the restriction of line bundle on fibers. 

Let us revisit the standard Bergman kernel on reduced complex space. For further information, refer to \cite{cmm17}. Consider a complex space $X$ with a Radon measure $\mu$ and a holomorphic line bundle $L$ equipped with a continuous Hermitian metric $h$. We define the {\it Bergman space} as the set of holomorphic sections $s$ of $L$ satisfying $\int_{X} \left\Vert s \right\Vert_{h}^2 d\mu < \infty$, denoted as $H^0_{L^2} \left( X ,L \right)$. The space is endowed with the $L^2$-inner product $\left\langle s_1 , s_2 \right\rangle_{L^2 ;X,L ,h} = \int_{X } \left\langle s_1 , s_2 \right\rangle_{h}^2 d\mu$, and the corresponding $L^2$-norm $\left\Vert s \right\Vert_{L^2 ;X,L ,h} = \left\langle s,s \right\rangle_{L^2 ;X,L,h}^{1/2}$. We use the shorthand $\left\langle s_1 , s_2 \right\rangle_{L^2}$ and $\left\Vert s \right\Vert_{L^2}$ for brevity. The {\it Bergman kernel function} is defined as 
$$ \rho_{X, \mu ,L,h} (x) = \sup \left\Vert s(x) \right\Vert_{h}^2, $$ 
where $x\in X$ and the supremum is taken over all $s\in H^0_{L^2} \left( X ,L \right)$ with $\left\Vert s \right\Vert_{L^2}^2=1$. We may also use the abbreviated notation $\rho_L(x)$ when the measure and Hermitian metric are understood.

We make some preparation before defining the fiberwise Bergman kernel functions. 

Consider normal complex spaces $\CX$ and $ C$. Let $\pi:\CX\to C$ be a proper surjective morphism with reduced general fibers. Suppose that $\CL$ is a line bundle on $\CX$ equipped with a continuous Hermitian metric $h$, and $\omega$ is a continuous $2$-form on $\CX$ whose restriction to each fiber is a Hermitian form. Denote the scheme-theoretical fiber of $\pi$ at $t\in C$ by $\CX_t=\pi^{-1}(t)$, whose reduction $\CX_{t,\red}$ is the set-theoretical fiber of $\pi$.  Assume that $\pi$ satisfies one of the following conditions:
\begin{itemize}
	\item $\pi $ is flat, and all fibers of $\pi $ are reduced;
	\item $C$ is a Riemann surface.
\end{itemize}
In the first case, the fiberwise Bergman kernel can be defined as the standard Bergman kernel on reduced complex spaces. We assume that $C$ is a Riemann surface. There may be a descrete set of points over which the fibers are non-reduced. In this case, the flatness of $\pi$ follows automatically by the following lemma. 

\begin{lem}
\label{lemmariemannsurfaceflat}
Let $\CX$ be a reduced and irreducible analytic variety, and $\pi:\CX\to C$ be a surjective holomorphic map to a Riemann surface $C$. Then $\pi$ is flat. 
\end{lem} 
\begin{proof}
The question is local on $C$. We may assume that $C=\D = \{z\in \IC: |z|^2 < 1\}$ is the unit disk, and $\CX_t$ is reduced for all $t\in \D^\circ=\D\setminus \{0\}$.  We denote by $\CO_\CX$ the holomorphic structure sheaf of $X$, and it suffice to show that the local ring $\CO_{\CX, x}$ is flat over $\IC[t]_{(t)}$ (the local ring of $\IC[t]$ at the ideal $(t)$) for any $x\in \pi^{-1}(0)$. Since $\CX$ is reduced and irreducible, $\CO_{\CX, x}$ is an integral domain. It is a torsion free $\IC[t]_{(t)}$ algebra, hence is flat.  
\end{proof}

Note that $\CX_0$ may be non-reduced. Since $\pi$ is flat, all the irreducible components of $\CX_0$ is of codimension one in $\CX$. Hence we have the decomposition $\CX_0=\sum_{j=1}^q m_j Y_j$ as an effective divisor on $\CX$, where $Y_j$ are the irreducible components of $\CX_0$. Let $f$ be a continuous function on $\CX_{0,\red }$, and $\mu$ a non-negative measure on $\CX_{0,\red}$. We define the integration with multiplicity by
$$\int_{\CX_0}fd\mu:=\sum_{j=1}^q m_j \int_{Y_j} fd\mu, $$
where $\int_{Y_j} fd\mu $ are the standard integration. If $f=\mathbf{1}_{A}$ for some subset $A\subset \CX_{0,\red}$, then we write $\int_{A} d\mu = \int_{\CX_0}fd\mu $. When $\CX_0$ is reduced, we see that all the $m_i$ are $1$, and the integration $\int_{\CX_{0}}fd\mu$ is consistent with the standard definition.

With the notations above, we have the following theorem of J. King.

\begin{thm}{\cite[Theorem 3.3.2]{kin71}}
\label{king continuous thm}
Let $f :\CX \to \IC$ be a continuous function, and let $\omega_t $ be the restriction of $\omega $ on $\CX_{t,\red}$. Then the function of fiberwise integration with multiplicity, $\varphi (t)=\int_{\CX_t}f \omega_t^n$, is continuous on $C$.  
\end{thm}

Fix $0\neq s\in H^0 (\CX_0 ,\CL_0 ) $. Then for any point $p\in \CX_{0,\red } $, there exists an open neighbourhood $U$ of $p$ in $\CX$ and a holomorphic section $\tilde{s} \in H^0 (U ,\CL ) $ such that $\tilde{s} |_{\CX_0 } = s$. If $U \cap Y_j \neq \varnothing $, then it is a prime divisor on $U$. We may define 
\begin{displaymath}
\ord_{Y_j}(s) = \left\{ \begin{aligned}
\quad\ord_{Y_j}(\tilde{s}) & \quad \textrm{ if } \ord_{Y_j}(\tilde{s}) \leq m_j -1 ; \\
\quad\infty\quad\;\; & \quad \textrm{ if } \ord_{Y_j}(\tilde{s}) \geq m_j,
\end{aligned} \right. 
\end{displaymath}
which is independent of the choice of $U$ and $\tilde{s}$. 
And we denote by
$$\ord_0(s) = \min_{1\leq j\leq q} \frac{\ord_{Y_j}(s)}{m_j}. $$
Since $s$ is non-zero in $H^0 (\CX_0 ,\CL_0 )$, we have $\ord_0(s) <1$. 
By reordering those $j$, we may assume that $\ord_0(s) = \frac{\ord_{Y_1}(s)}{m_1} $. Hence $ t^{-\ord_{Y_1} (s)}\tilde{s}^{m_1}$ is a non-zero holomorphic section in $H^0 (U ,\CL^{m_1})$, and we have a continuous function 
$$h(t^{-\ord_{Y_1}(s)}\tilde{s}^{m_1})^{\frac{1}{m_1}} = |t|^{-2\ord_0(s)}h(\tilde{s}), $$
on $U$. Hence for any $0\le \alpha \le \ord_0(s)$, we may define
$$\tilde{h}_{0,\alpha}(\tilde{s})(x) = |\pi(x)|^{-2\alpha}h(\tilde{s})(x), x\in U, $$
which is a continuous function on $U$.
Note that the restriction of this function on $U\cap \CX_{0,\red}$ is independent of the choice of $\tilde{s}$. We get a continuous function on $U\cap \CX_{0,\red}$ by 
$$h_{0, \alpha}(s)=\tilde{h}_{0,\alpha}(\tilde{s})|_{U\cap \CX_{0,\red}}. $$
By taking an open cover, we see that the function $h_{0,\alpha}(s)$ is defined on the whole $\CX_{0, \red}$. 
The function $h_{0,\alpha}(s)$ is identically zero on $\CX_{0,\red}$ when $\ord_0(s) > \alpha$, and is non-zero on some irreducible component of $\CX_{0,\red}$ when $\ord_{0}(s)=\alpha$.

%For another non-zero section $s'\in H^0(\CX_0, \CL_0)$, we may shrink $U$ and find a section $\tilde{s'}\in H^0(U, \CL)$ whose restriction on $U\cap \CX_0$ is $s'$. 

%It's clear that the function is also independent of the choice of $U$. By choosing an open cover, we see that the function $h_{0,\alpha}(s)$ is defined on $\CX_{0,\red}$. 

Similarly, we may define a function $h_{0,\alpha}(s,s')$ on $\CX_{0,\red}$ for $s,s' \in H^0(\CX_0, \CL_0)$ with $\ord_0(s),$ $\ord_0(s') \ge \alpha$. For any point $p\in \CX_{0, \red}$ there exists open neighbourhood $U$ of $p$ in $\CX$ such that $s, s'$ extends to holomorphic sections $\tilde{s}, \tilde{s}'$ on $U$ respectively. We first define 
\begin{eqnarray*}
2\tilde{h}_{0,\alpha} (\tilde{s},\tilde{s}') (x) & = & 
\tilde{h}_{0,\alpha}(\tilde{s}+\tilde{s}')(x)-\tilde{h}_{0,\alpha}(\tilde{s})(x)-\tilde{h}_{0,\alpha}(\tilde{s}')(x)
 \\
& & + \sqrt{-1} \left( \tilde{h}_{0,\alpha}(\tilde{s}+\sqrt{-1}\tilde{s}')(x)-\tilde{h}_{0,\alpha}(\tilde{s})(x)-\tilde{h}_{0,\alpha}(\tilde{s}')(x) \right) ,
\end{eqnarray*}
for $x\in U$. The restriction of such a function on $U\cap \CX_{0,\red}$ is also independent of the choice of $\tilde{s}, \tilde{s}'$. Then we denote by
$$h_{0,\alpha}(s,s'):=\tilde{h}_{0,\alpha} (\tilde{s},\tilde{s}')|_{U\cap \CX_{0, \red}}, $$
which is indeed defined on the whole $\CX_{0, \red}$. It's clear that $h_{0,\alpha}(s,s)=h_{0,\alpha}(s)$. The function $h_{0,\alpha}(s, s')$ is identically zero if $\ord_0(s)+\ord_0(s')>2\alpha$.

The function $\ord_0$ induces a descending filtration $\CF$ on $H^0 (\CX_0 ,\CL_0 )$, which is defined by 
$$\CF^{\lam}H^0 (\CX_0 ,\CL_0 ) = \left\lbrace s\in H^0 (\CX_0 ,\CL_0 ) : \ord_0(s) \geq\lambda \right\rbrace, \lam\in\IR.$$
There are rational numbers $0\leq \lambda_1 <\cdots <\lambda_l <1$ such that the locally constant function $\lam \mapsto \dim \CF^{\lam}H^0 (\CX_0 ,\CL_0 ) $ only jumps at those $\lambda_i$. Hence the image of the map $\ord_0 : H^0 (\CX_0 ,\CL_0 ) \to [0,\infty ] $ is just $\{ \lambda_1 ,\cdots ,\lambda_l ,\infty \}$. Write $\lambda_{l+1} =1 $, then $\CF^{\lam_{l+1}}H^0 (\CX_0 ,\CL_0 ) =0 $.

By definition, the pairing $h_{0,\lam_i}(-,\sim)(x)$ is a Hermitian form on $\CF^{\lam_i}H^0 (\CX_0 ,\CL_0)$ for any $x\in \CX_{0, \red}$, which is only semi-positive definite. 
One can see that $h_{0,\lam_i} (s,s') (x) = 0$ when $s$ or $s'\in \CF^{\lam_{i+1}}H^0 (\CX_0 ,\CL_0 ) $. Hence $h_{0,\lambda_i } (-,\sim) (x)$ is actually a Hermitian form on $\gr_\CF^{\lam_i} H^0 (\CX_0 ,\CL_0 ) $, where 
$$\gr_\CF^{\lam_i} H^0 (\CX_0 ,\CL_0 ) = \CF^{\lam_i}H^0 (\CX_0 ,\CL_0 ) / \CF^{\lam_{i+1}}H^0 (\CX_0 ,\CL_0 ). $$
Let $\gr_\CF H^0 (\CX_0 ,\CL_0 ) = \oplus_{i=0}^l \gr_\CF^{\lam_i} H^0 (\CX_0 ,\CL_0 ) $. For any $x\in\CX_{0,\red }$, we denoted by
$$h_0 (- ,\sim ) (x) := \sum_{i=1}^l h_{0,\lambda_i } (-,\sim) (x), $$ 
which is a Hermitian form on $\gr_\CF H^0 (\CX_0 ,\CL_0 )$. 

For any $s\in \gr_\CF H^0 (\CX_0 ,\CL_0 ) \backslash \{ 0 \}$, we have $h_0 (s,s) (x)> 0$ for some $x\in\CX_{0,\red} $. Hence the integral
\begin{displaymath}
\left\langle s,s' \right\rangle_{L^2 ;\CX_0 } = \int_{\CX_0 } h_0 (s,s') \omega_0^n, 
\end{displaymath}
defines a Hermitian metric on $\gr_\CF H^0 (\CX_0 ,\CL_0 )$. And the graded decomposition 
$$\gr_\CF H^0 (\CX_0 ,\CL_0 ) = \oplus_{i=0}^l \gr_\CF^{\lam_i} H^0 (\CX_0 ,\CL_0 ),$$ 
is orthogonal with respect to this Hermitian metric. We simply denoted by $\Vert s \Vert_{L^2 ;\CX_0 } = ( \left\langle s,s \right\rangle_{L^2 ;\CX_0 } )^{\frac{1}{2}} $.

Now we are ready to state the definition of fiberwise Bergman kernel on $\CX_0 $.

\begin{defi} \rm
\label{definitionfiberwise bergmankernel}
The {\it fiberwise Bergman kernel} of the system $(\CX, \CL, h, \omega )$ on $\CX_0 $ is defined to be the function on $\CX_{0,\red }$
$$\rho_{\CL_0} (x)=\sup \lVert s(x) \rVert_{h_0}^2,\quad x\in \CX_{0,\red } , $$
where the supremum runs over all $s\in \gr_\CF H^0(\CX_0, \CL_0 )$ with $\lVert s \rVert_{L^2 ;\CX_0 } = 1$. 
\end{defi}

\begin{rmk} \rm
If $\CX_0$ is reduced, $0\neq s\in H^0 (\CX_0 ,\CL_0 ) $ implies that $\ord_0 (s) =0 $, $h_0 =h $, and $\rho_{\CL_0}$ is just the standard Bergman kernel on the reduced complex space. Hence the fiberwise Bergman kernel can be viewed as a generalization of the standard Bergman kernel.
\end{rmk}

Let $\{s_j\}_{1\le j\le N}$ be an $L^2$-orthonormal basis of $\gr_\CF H^0(\CX_0, \CL_0 )$. Then we have 
$$\rho_{\CL_0} (x)=\sum_{j=1}^N \lVert s_j(x) \rVert_{h_0}^2. $$ 
Indeed, for any $s\in \gr_\CF H^0(\CX_0, \CL_0 )$ with $\Vert s \Vert_{L^2 ;\CX_0 }^2=1$, we have $s=\sum_ja_js_j$ for some $a_j\in\IC$ and $\sum_j|a_j|^2=1$. Then by Schwarz lemma we have 
$\lVert s(x) \rVert_{h_0}^2 \le\sum_j\lVert s_j(x) \rVert_{h_0}^2$, and the equality holds if and only if $a_js_i(x)=a_is_j(x)$ for all $i,j$.

As a consequence, we have $\int_{\CX_0 } \rho_{\CL_0} \omega_0^n = \dim_{\mathbb{C}} \gr H^0(\CX_0, \CL_0 ) = h^0(\CX_0, \CL_0 )$. By Theorem \ref{king continuous thm}, we can obtain the following result.

\begin{cor}
\label{bgmc0imH0constant}
Let $(\CX,\CL,\omega, h, \pi, C)$ be as above. Assume that either $C$ is a Riemann surface or all fibers of $\pi$ are reduced. 
If the fiberwise Bergman kernel $\rho_{\CL_{\pi (x)}} (x)$ is continuous. Then the function $h^0(\CX_t, \CL_t)=\dim H^0(\CX_t, \CL_t)$ is constant for all $t\in C$.
\end{cor}

\section{The continuity of fiberwise Bergman kernels}
\label{section_continuity}

We will prove Theorem \ref{maintheorem} in this section. The key ingredient is the following theorem. 

\begin{thm}[Grauert] \label{Grauertuppersemicontinuous}
Let $\pi :\CX\to C$ be a proper flat morphism between reduced complex spaces. For any line bundle $ \CL $ on $\CX$, the function $t \mapsto h^0 (\CX_t ,\CL_t ) $ is upper semi-continuous on $C$ in Zariski topology. 
Moreover, if $h^0 (\CX_t ,\CL_t )$ is constant for $t\in C$, then the coherent sheaf $\pi_*\CL $ is locally free, and the restriction map induces an isomorphism $ (\pi_*\CL)_t / \mathfrak{m}_t (\pi_*\CL)_t \cong H^0(\CX_t, \CL_t)$ for any $t\in C$.
\end{thm}

\begin{proof}
See \cite[Section 10.5]{gs84} or \cite{kiv1} for details. An algebraic version of this theorem can be found in \cite[Theorem 3.12.8]{har77}.
\end{proof}

As a consequence, we have the following lemma.

\begin{lem} \label{Grauertlemma}
Let $\pi :\CX\to C$ be a proper flat morphism between reduced complex spaces. For any line bundle $ \CL $ on $\CX$,  the following properties are equivalent:
\begin{enumerate}
\item The function $h^0 (\CX_t ,\CL_t ) $ is constant for $t\in C$.
\item For any $t_0 \in C$, there are holomorphic sections $\tilde{s}_1, \cdots, \tilde{s}_N$ $\in H^0(\pi^{-1} (U ), \CL)$, where $U$ is an open neighbourhood of $ t_0 $ in $C$, such that the restriction map induces an isomorphism $ \mathbb{C} \la\tilde{s}_1, \cdots, \tilde{s}_N \ra \cong H^0(\CX_{t }, \CL_{t } )$ for any $t\in U $.
\end{enumerate}
\end{lem}

\begin{proof}
$(2) \Rightarrow (1) $ is trivial. Now we consider $(1) \Rightarrow (2) $. 

Assume that $h^0 (\CX_t ,\CL_t ) $ is constant. By Theorem \ref{Grauertuppersemicontinuous}, we see that $\pi_*\CL $ is a holomorphic vector bundle on $C$, and for any $ t\in C$, the restriction map gives an isomorphism $ (\pi_*\CL)_t / \mathfrak{m}_t (\pi_*\CL)_t \cong H^0(\CX_t, \CL_t)$. Fix $t_0 \in C $. Then there exist an open neighbourhood $U $ of $t_0 $ and a frame $ \tilde{s}_1, \cdots, \tilde{s}_N$ $\in H^0( U , \pi_* \CL) \cong H^0(\pi^{-1} (U ), \CL) $. Since the restriction map $ (\pi_*\CL)_t / \mathfrak{m}_t (\pi_*\CL)_t \to H^0(\CX_t, \CL_t)$ are isomorphisms, we conclude that for any $t\in U$, the restriction map induces an isomorphism $ \mathbb{C} \la\tilde{s}_1, \cdots, \tilde{s}_N \ra \cong H^0(\CX_{t }, \CL_{t } )$.
\end{proof}

In order to demonstrate  Theorem \ref{maintheorem}, it is necessary to utilize the subsequent lemma.

\begin{lem} 
\label{gluing lemma}
Let $\CX$ be a normal complex space, $\omega $ be a smooth Hermitian metric on $\CX_\reg$, and $\pi:\CX\to C$ be a surjective proper flat morphism with reduced general fibers, where $C$ is a Riemann surface. Let $\CX_{t_0} =\sum_{j=1}^q m_j Y_j$, where $t_0 \in C $, and $Y_j$ are the irreducible components of $\CX_{t_0} $. Let $K$ be a compact subset of $ \CX_{0,\red,\reg }\setminus \CX_\sing $. Then there exists an open neighbourhood $\D_{t_0 } $ of $t_0 \in C$ satisfying the following properties. 

For any $t\in\D_{t_0} $, there exist an open subset $\Omega_{K,t} \subset \CX_{t,\red,\reg }\setminus \CX_\sing $ and a non-degenerate $C^\infty $ map $F_{K,t} : \Omega_{K,t} \to \CX_{0,\red,\reg }\setminus \CX_\sing $ such that the image of $F_{K,t}$ containing $K$, $\sup_{x\in \Omega_{K,t} } d_{\omega } (x,F_{K,t} (x) ) \to 0 $ as $t\to t_0 $, $ 
\Vert F_{K,t}^{*} \omega_{t_0} - \omega_t \Vert_{C^m ,\omega_t} \to 0 $ as $t\to t_0$, $\forall m\in\mathbb{N} $, and $F^{-1}_{K,t} (x) $ has $m_j$ points when $x\in Y_j \cap K $, $j=1,\cdots ,q$, where $\omega_t $ is the restriction of $\omega $ on $\CX_{t,\red,\reg }\setminus \CX_\sing $.
\end{lem}

Loosely speaking, this means that there are maps $F_{K,t}$ between big open subsets of the general fibers and the central fiber inducing metric convergence as $t\to 0$. 

\begin{proof}
Similar to the argument in \cite[Theorem 11.3.6 (C)]{pet16}, we use the standard gluing program to prove the lemma.

The question is local on $C$, we just assume that $C=\D$, $t_0 =0$, and $\CX_t$ is reduced for all $t\in \D^\circ=\D\setminus \{0\}$. 
Let $\CX_{0} = \cup_{j=1}^q Y_j$ be the irreducible decomposition. Since $\pi$ is flat, we see that $Y_j$ are prime divisors on $\CX$. Note that $\CX_0$ may be non-reduced. We may assume that $Y_j$ has multiplicity $m_j$ in $\CX_0$.

Let $Y_{j,\reg}:=Y_j\setminus (\CX_{0,\red,\sing }\cup \CX_\sing)$. Since $\CX$ is normal, $Y_{\reg} = \CX_{0,\red,\reg }\setminus \CX_\sing = \amalg_j Y_{j,\reg}$ is a disjoint union of open dense subsets of $\CX_{0,\red}$. 

For any point $y\in Y_{j,\reg}$, there exist a coordinate chart $U_y$ of $y$ in $\CX$ with coordinates $z_0,\cdots,z_n$ such that $\pi(z_0,\cdots, z_n)=z_0^{m_j}$. Hence $U_y \cap \CX_t=\amalg_{k=1}^{m_j} U_{y,t}^k$ is a disjoint union of $m_j$ open subsets which are biholomorphic to each other, and we have biholomorphic maps $f_{t}^k :U_{y,t}^k \to U_{y,0} = U_y\cap Y,\; 1\le k\le m_j$ for $t\in\D^\circ $ with sufficiently small norm. We have the metric convergence $\lVert f_t^{k*}\omega_t-\omega_0 \rVert_{C^m ,\omega_t}\to 0$ as $t\to 0$ for any $m\ge2$. 

For another coordinate chart $V$ of $y$ in $\CX$, similarly we have $V\cap \CX_t=\amalg_{k=1}^{m_j}V_t^k$ and biholomorphic maps $g_{t}^k :V_t^k\to V_0, 1\le k\le m_j$. We may assume that $U_t^k\cap V_t^k\ne \varnothing$. The restrictions $f_t^k|_{U_t^k\cap V_t^k}$ and $g_t^k|_{U_t^k\cap V_t^k}$ may be different. Instead, we have $\lVert f_t^k-g_t^k\rVert_{C^\infty } \to 0$ as $t\to 0$. More explicitly
$$\lVert \zeta^{-1}\circ f_t^k|_{U_t^k\cap V_t^k}\circ\eta 
 - \zeta^{-1}\circ g_t^k|_{U_t^k\cap V_t^k}\circ \eta \rVert_{C^{m+1} , \omega_{\rm Euc} } \to 0, $$
as $t\to 0$, for any $m\ge2$, and some (hence any) holomorphic coordinate charts $\eta: \Delta^{n}\to U_t^k\cap V_t^k$ and $\zeta: \D^n\to U_0\cap V_0$. We just say that $f_t^k$ and $g_t^k$ are {\it arbitrarily close}.

We may choose a sequence of such coordinate charts $\{U_s^j\}_{1\le j\le q, s\in \IN}$ with corresponding biholomorphic maps $f_{ts}^{jk}:U_{ts}^{jk} \to U_{0s}^{j}$ where $U_{0s}^{j}=U_s^j\cap Y_j$ and $U_{ts}^{jk}\seq U_s^k\cap \CX_t$ inducing the metric convergence $\lVert f_{ts}^{jk*}\omega_0-\omega_t\rVert_{C^m} \to 0$. We denote by 
$$\Omega_{l}^{j}=\cup_{s=1}^l U_{s}^j \seq \CX, \quad
\Omega_{0l}^{j}=\cup_{s=1}^l U_{0s}^j \seq Y_j, \quad
\Omega_{tl}^{jk}=\cup_{s=1}^l U_{ts}^{jk}\seq \CX_t$$
the corresponding ascending chains of open subsets, and assume that $K\subset 
\cup_{j=1}^q \Omega_{0l}^j $ for some $l\in\mathbb{N}$. We may assume that, for any $l\in\IN$, the open subsets $\Omega_l^j (1\le j\le q)$  are disjoint. We may also assume that, there exist $\delta >0$, for any $j,l$ and $t$ with $|t|\le\delta$, $\Omega_{tl}^{jk} (1\le k\le m_j)$ are disjoint.

For any given integer $1\le j\le q$, we may choose the lexicographic order on the parameter set $\{(s,k)\}_{1\le k\le m_j, s\in \IN}$ of $\{ f_{ts}^{jk} \}_{1\le k\le m_j, s\in \IN}$, that is, $(s,k)>(s',k')$ if $s>s'$ or $s=s', k>k'$. We denote by $\tilde{\Omega}_{tl}^{jk}=\cup_{(s,p)\le(l,k)}U_{ts}^{jp}$. Note that $\tilde{\Omega}_{tl}^{jm_j} = \cup_{k=1}^{m_j}\Omega_{tl}^{jk}$. 

Now we define the maps $\tilde{F}_{tl}^{jk}:\tilde{\Omega}_{tl}^{jk}\to \Omega_{tl}^{j}$ by induction on $(l,k)$, whose restriction on $\Omega_{tl}^{jk}$ will be the $F_{tl}^{jk}$. For $(l,k)=(1,1)$ we simply define $F_{t1}^{j1}=f_{t1}^{j1}$. Suppose we have defined $\tilde{F}_{tl_0}^{jk_0}$ for all $(l_0,k_0)<(l,k)$ for $|t|<\delta$, where $\delta>0$ is a constant depend on $(l,k)$. We denote by $(l',k')$ the pair one less than $(l, k)$. 
If $U_{tl}^{jk}\cap\tilde{\Omega}_{tl'}^{jk'} = \varnothing$, then we just define $\tilde{F}_{tl}^{jk}=\tilde{F}_{tl'}^{jk'}$ on $\tilde{\Omega}_{tl'}^{jk'}$ and $\tilde{F}_{tl}^{jk}=f_{tl}^{jk}$ on $U_{tl}^{jk}$. 
If $U_{tl}^{jk}\seq\tilde{\Omega}_{tl'}^{jk'}$, we define $\tilde{F}_{tl}^{jk}=\tilde{F}_{tl'}^{jk}$. Otherwise, we define $\tilde{F}_{tl}^{jk}$ on $\tilde{\Omega}_{tl}^{jk}=U_{tl}^{jk}\cup \tilde{\Omega}_{tl'}^{jk'}$ by 
$$\tilde{F}_{tl}^{jk}(x)=\varphi_{0l}^{j}\circ\big(\mu_1(x)\cdot (\varphi_{0l}^{j})^{-1}\circ f_{tl}^{jk}(x)+\mu_2(x)\cdot (\varphi_{0l}^{j})^{-1}\circ \tilde{F}_{tl'}^{jk'}(x) \big), $$
where $\{\mu_1, \mu_2\}$ is a partition of unity for $U_{tl}^{jk}, \tilde{\Omega}_{tl'}^{jk'}$, and $\varphi_{0l}^j:\D^n\to U_{0l}^j$ is a coordinate chart. We simply write
\begin{displaymath}
\left\{ \begin{array}{ccc}
F_0
&:=& (\varphi_{0l}^{j})^{-1}\circ \tilde{F}_{tl}^{jk} \circ \varphi_{tl}^{jk}, \\
F_1
&=& (\varphi_{0l}^{j})^{-1}\circ f_{tl}^{jk} \circ \varphi_{tl}^{jk}, \\
F_2
&:=& (\varphi_{0l}^{j})^{-1}\circ \tilde{F}_{tl'}^{jk'} \circ \varphi_{tl}^{jk}, \\
\end{array} \right.
\quad
\left\{ \begin{array}{ccc}
\tilde{\mu}_1
&:=& \mu_1\circ\varphi_{tl}^{jk}, \\
\tilde{\mu}_2
&:=& \mu_2\circ\varphi_{tl}^{jk}, 
\end{array}\right. 
\end{displaymath}
where $\varphi_{tl}^{jk}:\D^n\to U_{tl}^{jk}$ is a coordinate chart. 
Hence $F_0=\tilde{\mu}_1F_1 + \tilde{\mu}_2F_2$, and
\begin{eqnarray*}
\lVert F_0 - F_1 \rVert_{C^{m+1} ,\omega_t }
&=& \rVert \tilde{\mu}_2(F_2-F_1)\rVert_{C^{m+1} ,\omega_t } \\
&\le& C_1 \cdot \lVert\tilde{\mu}_2\rVert_{C^{m+1}}\cdot \lVert F_2-F_1\rVert_{C^{m+1} ,\omega_t }, 
\end{eqnarray*}
where $C_1 =C_1 (m) >0 $ is a constant. 
Since $f_{ts}^{jp} ((1,1)\le (s,p) \le (l',k'))$ are arbitrarily close, we see that $F_{tl'}^{jk'}$ and $f_{tl}^{jk}$ are arbitrarily close. Hence $\lVert F_2-F_1\rVert_{C^{m+1} ,\omega_t}\to 0$ as $t\to 0$. We conclude that $F_{tl}^{jk}$ is arbitrarily close to $f_{ts}^{jp}$ for $(1,1)\le (s,p) \le (l,k)$. Let $F_{K,t}$ be a holomorphic map on $ \Omega_{K,t} = \cup_{j=1}^{q} \tilde{\Omega}_{tl}^{jk}$ such that $F_{K,t} =F_{tl}^{jk} $ on $\tilde{\Omega}_{tl}^{jk}$. By the construction, one can easily see that $\sup_{x\in \Omega_{K,t} } d_{\omega } (x,F_{K,t} (x) ) \to 0 $ as $t\to 0 $. Then the proof is complete.
\end{proof}

\begin{rmk} \rm
By the argument above, for any $m\in\mathbb{Z}_{\geq 0} $ and $\alpha \in [0,1 ]$, we conclude that $ \Vert F_{K,t}^{*} \omega_{t_0} - \omega_t \Vert_{C^{m,\alpha } ,\omega_t} \to 0 $ as $t\to t_0$ when $\omega $ is a $C^{m,\alpha}$ 2-form on $\CX $ such that the restriction $\omega_t $ is a $C^{m,\alpha}$-Hermitian metric on $\CX_t $.
\end{rmk}

We are ready to prove Theorem \ref{maintheorem}. 

\begin{thm}[={Theorem \ref{maintheorem}}]
Let $\pi:\CX\to C$ be a proper surjective morphism with reduced general fibers between normal complex spaces. Suppose that $\CL$ is a line bundle on $\CX$ equipped with a continuous Hermitian metric $h$, and $\omega$ is a continuous $2$-form on $\CX$ whose restriction to each fiber is a Hermitian form. We assume that either $C$ is a Riemann surface or all fibers of $\pi$ are reduced. 
Then the following statements are equivalent:
\begin{enumerate}
\item The function $h^0(\CX_t, \CL_t)=\dim H^0(\CX_t, \CL_t)$ is constant for $t\in C$.
\item The fibrewise Bergman kernel function $\rho_{\CL_{\pi (x)} } (x)$ is continuous on $\CX$.
\end{enumerate}
\end{thm}

\begin{proof}
$(2) \Rightarrow (1) $ follows from Corollary \ref{bgmc0imH0constant}. Now we prove $(1) \Rightarrow (2) $.

We first consider the case that all the fibers of $\pi$ are reduced. Fix $t_0 \in C $. Since $h^0(\CX_t, \CL_t)=\dim H^0(\CX_t, \CL_t)$ is constant, we can use Lemma \ref{Grauertlemma} to find a neighbourhood $\CU $ of $t_0 \in C $ and holomorphic sections $\tilde{s}_1, \cdots, \tilde{s}_N$ $\in H^0(\pi^{-1} (\CU ), \CL)$ such that for any $t\in\CU $, the restriction map induces an isomorphism $ \mathbb{C} \la\tilde{s}_1, \cdots, \tilde{s}_N \ra \cong H^0(\CX_t, \CL_t )$.

By taking $f_{ij}=h( \tilde{s}_i, \tilde{s}_j)$, which is continuous on $\CX$, we see that the functions
$a_{ij}(t)=\int_{\CX_t} f_{ij} \omega_t^n$
are continuous for any $1\le i,j\le N$ by Proposition \ref{king continuous thm}. Clearly, the isomorphism $ \mathbb{C} \la\tilde{s}_1, \cdots, \tilde{s}_N \ra \cong H^0(\CX_{t }, \CL_{t } )$ implies that the matrices $(a_{i,j} (0) )_{1\leq i,j\leq N}  > 0$ on $\CU$. By the Schmit orthogonalization program, we can find continuous functions $b_{ij}(t)$ on $\CU$ such that $\{\hat{s}_{i}=\sum_j b_{ij}(t) \tilde{s}_j\}_{1\le i\le N}$ is an $L^2$-orthonormal basis of $H^0(\CX_t, \CL_t)$ for any $t\in \CU $. We conclude that the function
$$\rho_{L_{\pi (x)}}(x)=\sum_kh(\hat{s}_k(x),\hat{s}_k(x))=\sum_{i,j,k}b_{ki}(\pi(x))\overline{b_{kj}(\pi(x))} h(\tilde{s}_i(x),\tilde{s}_j(x))$$
is continuous on $\CU $, which proves Theorem \ref{maintheorem} in the case $(\alpha)$.

Next we consider the case that $C$ is a Riemann surface. The argument is almost the same as above, but we need to deal with non-reduced fibers here. The question is local on $C$. We assume that $C=\D$ and $\CX_t$ is reduced for all $t\ne 0$.

By Lemma \ref{Grauertlemma}, the constancy of $h^0(\CX_t, \CL_t)=\dim H^0(\CX_t, \CL_t)$ shows that there are holomorphic sections $\tilde{s}_1, \cdots, \tilde{s}_N$ $\in H^0(\CX, \CL)$ such that for any $t\in\D $, the restriction map induces an isomorphism $ \mathbb{C} \la\tilde{s}_1, \cdots, \tilde{s}_N \ra \cong H^0(\CX_t, \CL_t )$. Note that the question is local on the base, and one can always shrink $\D$ if necessary.

Write $h^0_{0,i} =\dim \CF^{\lam_i}H^0 (\CX_0 ,\CL_0 ) $, where $\CF^{\lam_i}H^0 (\CX_0 ,\CL_0 ) $ is the linear subspace of $H^0 (\CX_0 ,\CL_0 ) $ defined in Section \ref{section_Fiberwise Bergman kernel}. Let $s_{j,t}$ denotes the restriction of $s_j$ on $\CX_t$. By taking a linear transformation, we assume that for any $ j\leq h_{0,i}^0 $, the 
restriction $s_{j,0} = \tilde{s}_j |_{\CX_0 } \in \CF^{\lam_i}H^0 (\CX_0 ,\CL_0 ) $, and the set $\{ s_{j,0} \}_{j= h_{0,i+1}^0 +1}^{h_{0,i }^0 } $ gives an $L^2$ orthonormal basis of $\gr_\CF^{\lam_i} H^0 (\CX_0 ,\CL_0 ) = \CF^{\lam_i}H^0 (\CX_0 ,\CL_0 ) /H_{\lambda_{i+1}}^0 (\CX_0 ,\CL_0 ) $. Hence $\{ s_{j,0} \}_{j =1 }^N $ gives an $L^2$ orthonormal basis of $\gr H^0 (\CX_0 ,\CL_0 ) $. By definition, we have $\rho_{L_0 } = \sum_{j=1}^{N} h_0 (s_{j,0} ,s_{j,0} ) $.

For $t\in\D$ and $i,k = 1,\cdots ,N$, we define
\begin{displaymath}
a_{ik} (t) = \left\{ \begin{aligned}
|t|^{-\ord_0 (s_{i,0}) - \ord_0 (s_{k,0}) }\int_{\CX_t} h ( s_{i,t} ,s_{k,t} ) \omega_t^n , & \textrm{ if } t\neq 0 ; \\
\int_{\CX_0} h_0 ( s_{i,0} ,s_{k,0} ) \omega_0^n , \;\;\quad\quad\quad & \textrm{ if } t=0 ,
\end{aligned} \right.
\end{displaymath}
where $s_{i,t}$ denotes the restriction of $s_i$ on $\CX_t$.

Since $t$ is continuous, we see that $a_{i,k} $ are continuous functions on $ \D^* = \D \setminus \{ 0 \} $.

By definition, for any given constant $\epsilon >0$, there exists a compact subset $K\subset \CX_{0,\red ,\reg} \setminus \CX_{\sing} $ such that the integral $\int_{ \CX_{0,\red } \setminus K } \omega_0^n < \epsilon $. Now we can apply Lemma \ref{gluing lemma} to find an open neighbourhood $\D_\delta $ of $0\in\D $ such that for any $t\in\D_{\delta} $, there are an open subset $\Omega_{K,t} \subset \CX_{t,\red,\reg }\setminus \CX_\sing $ and a non-degenerate $C^\infty $ map $F_{K,t} : \Omega_{K,t} \to \CX_{0,\red,\reg }\setminus \CX_\sing $ such that the image of $F_{K,t}$ containing $K$, $ \Vert F_{K,t}^{*} \omega^n_{0} - \omega^n_t \Vert_{C^0 ,\omega_t} \to 0 $ as $t\to 0$, and $F^{-1}_{K,t} (x) $ has $m_j$ points when $x\in Y_j \cap K $, $j=1,\cdots ,q$, where $\omega_t $ is the restriction of $\omega $ on $\CX_{t,\red,\reg }\setminus \CX_\sing $. Clearly, for any $t\in\D_\delta $ and continuous function $f$ on $\CX_t $, we have
$$\int_{K} \sum_{y\in F_{K,t}^{-1} (x) } f(y) \omega_0^n (x) = \int_{F_{K,t}^{-1} (K)} f (y) \left( F^{*}_{K,t} \omega_0 \right)^n  (y) .$$

For any $x\in K\cap Y_j$, there exists a coordinate chart $U_x$ of $x$ in $\CX$ with coordinates $z_0,\cdots,z_n$ such that $\pi(z_0,\cdots, z_n)=z_0^{m_j}$. Applying Lemma \ref{gluing lemma} again, one can see that 
$$\mathop{\lim}\limits_{t\to 0} \, \mathop{\sup}\limits_{ y\in F^{-1}_{K,t} (x) } \sum_{i=0}^n |z_i (x) - z_i (y) | =0 .$$
Now we conclude that the function 
$$\psi_x (t) = |t|^{-m_j^{-1} (\ord_{Y_j} (s_{i,0}) + \ord_{Y_j} (s_{k,0}) ) } \sum_{y\in F_{K,t}^{-1} (x) } h(s_{i,t} (y) , s_{k,t} (y) ) $$
can be extended to be a function on $\D_\delta $ such that $\psi_x $ is continuous at $0\in\D $. It follows that
$$ \mathop{\lim}\limits_{t\to 0} |t|^{-\ord_0 (s_{i,0}) - \ord_0 (s_{k,0}) }\int_{F_{K,t}^{-1} (K) } h ( s_{i,t} ,s_{k,t} ) \omega_t^n  = \int_{K} h ( s_{i,0} ,s_{k,0} ) \omega_0^n .$$

Since $\int_{ \CX_{0,\red } \setminus K } \omega_0^n < \epsilon $, we have $\lim_{t\to 0} \int_{ \CX_{t,\red } \setminus F_{K,t}^{-1} (K) } \omega_t^n < \epsilon $, hence there exists a constant $C>0$ which is independent of $\epsilon $, such that 
$$\mathop{\limsup}\limits_{t\to 0} \int_{ \CX_{t,\red } \setminus F_{K,t}^{-1} (K) } |t|^{-\ord_0 (s_{i,0}) - \ord_0 (s_{k,0}) } h ( s_{i,t} ,s_{k,t} ) \omega_t^n < C\epsilon .$$
Letting $\epsilon \to 0 $, we see that $a_{ik}$ are continuous on $\D$. 

As in the proof in case $(\alpha )$, we can use Schmit orthogonalization program to find continuous functions $b_{ij}(t)$ on $\D$ such that the set $\{\hat{s}_{i}=\sum_j b_{ij}(t) \tilde{s}_j\}_{1\le i\le N}$ is an $L^2$-orthonormal basis of $H^0(\CX_t, \CL_t)$ for any $t\in \D $, and $b_{ij} (0) =\delta_{ij} $. It follows that the function
\begin{displaymath}
\rho_{L_{\pi (x)}}(x)  = \left\{ \begin{aligned}
\sum_{i,j,k}b_{ki}(\pi(x))\overline{b_{kj}(\pi(x))} h_{\pi (x)} (\tilde{s}_i(x),\tilde{s}_j(x)) , & \textrm{ if } \pi (x)\neq 0 ; \\
\sum_k h_0 (s_{k,0}(x),s_{k,0}(x)) , \;\;\quad\quad\quad & \textrm{ if } \pi (x)=0 ,
\end{aligned} \right.
\end{displaymath}
is a continuous function on $\CX $, and the proof is completed.
\end{proof}

\section{Uniform convergence of Fubini-Study currents}
\label{section_uniform convergence}

We will prove Theorem \ref{uniformly convergencechap1} in this section. Let $\pi:\CX\to C$ be a proper surjective morphism, where $\CX$ is a normal complex space and $C$ is a Riemann surface (hence $\pi$ is flat by Lemma \ref{lemmariemannsurfaceflat}). Let $\omega $ be a continuous $2$-form on $\CX$ whose restriction on each fiber is a Hermitian metric.
Consider Hermitian line bundles $\{F_j\}_{1\le j\le k}$ on $\CX$ equipped with continuous Hermitian metrics $h_j$. We denote by $h_{j,t}$ the restriction of $h_j$ on $\CX_t$. Assume that $c_1(h_{1,t})\ge \varepsilon\omega_t$ for some uniform constant $\varepsilon>0$ with respect to all $t\in C$, and $c_1(h_{j,t})\ge 0 $ for all $t\in C$ and $2\le j\le k$. For any $1\le j \le k$, let $\{p_{j,m}\}$ be a sequence of positive integers with $\lim_{m\to\infty} \frac{p_{j,m}}{m}=r_j > 0$. Denote $\CL_m = F_1^{p_{1,m}} \otimes \cdots \otimes F_k^{p_{k,m}}$ and $h^{\CL_m}=\prod_{j=1}^k h_{j}^{p_{j,m}}$. The fiberwise Bergman kernel function of $(\CL_m, h^{\CL_m})$ on $\CX$ is simply denoted as $\rho_m(x)=\rho_{\CX_t, \CL_{m,t}}(x)$ for $x\in \CX$ and $t=\pi(x)$. 

\begin{thm} \label{uniformly convergence section 4}
The series of functions 
$$\varphi_m(t):=\int_{\CX_t}\frac{1}{m}
|\log(\rho_m)|\omega_t^n$$
converges to $0$ uniformly on compact subsets of $C$ as $m\to\infty$. 
\end{thm}

\begin{cor}
For any smooth $(n-1,n-1)$-form $\alpha$ on $\CX$, we have 
$$\mathop{\lim}_{m\to \infty} 
\int_{\CX_t}\left( \sum_{j=1}^k r_j c_1(h_{j,t}) - 
\frac{1}{m}\omega_{FS,m,t} \right)
\wedge \alpha \to 0 $$
uniformly on compact subsets of $C$ as $m\to \infty$. 
\end{cor}
\begin{proof}
Since $\sum_{j=1}^k r_j c_1(h_{j,t}) -\frac{1}{m}\omega_{FS,m,t}=- \frac{ 1 }{m} \iddbar \log(\rho_m|_{\CX_t})$, we have 
$$\int_{\CX_t}\left( \sum_{j=1}^k r_j c_1(h_{j,t}) - 
\frac{1}{m}\omega_{FS,m,t} \right)
\wedge \alpha_t
= 
-\int_{\CX_t} \frac{1}{m}
\log(\rho_m|_{\CX_t})
\cdot \iddbar\alpha_t. $$
Note that for any compact subset $K$ of $C$, $|\iddbar \alpha_t|_{\omega_t} \le C_K$ for some constant $C_K>0$, which is independent of $t\in K$. We conclude by applying the above theorem. 
\end{proof}

We will prove Theorem \ref{uniformly convergence section 4} in several steps. Since this question is local on $C$, we assume that $C=\D$ from now on.

{\it Step 1}. Upper bound on the regular part. We denote by $\CX_\reg$ the smooth locus of $\CX$ minus the singluarities of the reduction of all the fibers, in other word, $\CX_\reg$ is the largest open subset of $\CX$ such that it is smooth and the reduction of any fiber $\CX_{\reg,t,\red}$ is smooth. 

\begin{prop} \label{upper bound of the regular part}
For any compact subset $K\seq \CX_\reg$, we have 
$$\mathop{\limsup}_{m\to \infty} \,
\mathop{\sup}_{x\in K}\frac{\log(\rho_m)}{m}\le 0. $$
\end{prop}
\begin{proof}
We may choose open coverings $\{W_i\seq \CX\}$ and $\{U_i\seq \CX\}$ of $K$ such that $\cup \overline{W}_i\seq U_i \seq \CX_\reg$, and there exist coordinates $z_{i,1},\cdots, z_{i,n+1}$ on $U_i $ such that $\pi(z_{i,1},\cdots, z_{i,n+1})=z_{i,n+1}^{k_i}$ for some $k_i$. 

For any $0<\delta<1$, one may construct $C^\infty$ Hermitian metric $h_{j, \delta}$ on a neighbourhood of $\overline{W}_i $ by using the modification of $h_j$ on $U_{i}$, such that $(1-\delta )h_{j,\delta}\le h_j \le (1+\delta )h_{j,\delta} $ and $c_1(h_{\delta ,t})\le \delta^{-1} \omega_t $ on each $\overline{W}_i \cap \CX_{t} $.
Hence we have 
$\rho_m\le C_{K,\delta } (1-\delta)^{-C_{r } m} (1+\delta)^{C_{r } m} \rho_m'$ on $U_i$, where $C_{r}>0$ is a constant depending only on $r_j $, $C_{K,\delta}>0$ is a constant depending only on $K$ and $\delta $, $\rho_m'$ is the fiberwise Bergman kernel defined using the Hermitian metric $\prod_{j=1}^k h^{p_{j,m}}_{j,\delta}$ and the K\"ahler form $ \sum_{j=1}^k \frac{p_{j,m}}{m} c_1(h_{j,\delta}) $. 

Since $h_{j,\delta}$ are smooth and $F_j |_{W_i} $ are trivial line bundles, we can use the boundedness of the curvature of $ c_1(h_{j,\delta ,t}) $ and the standard Tian-Yau-Zelditch expansion to show that there exists a constant $C_{K,\delta }>0$ depending only on $K$ and $\delta $ such that $\rho_m' \le C m^n$ on $K$, $\forall m\in\mathbb{N} $. Hence we have
\begin{eqnarray*}
\mathop{\sup}_{x\in K}\frac{\log(\rho_m)}{m} & \le &
 \frac{\log(C_{K,\delta } (1-\delta)^{- C_{ r } m} (1+\delta)^{C_{ r } } m^n)}{m} \\
 & \le & C_{ r} \log (1+\delta ) - C_{ r} \log (1-\delta ) + \frac{\log(C_{K,\delta } m^n)}{m} .
\end{eqnarray*}
Then we get this proposition by taking $\limsup_{m\to \infty}$ and $\delta\to 0$. 
\end{proof}

{\it Step 2}. Global upper bound. We will bound $\frac{\log(\rho_m)}{m}$ from above near $\CX\setminus \CX_\reg$ by the upper bound of the regular part and using maximal principle. We need the following lemma to avoid the singularities when applying maximal principle. 

\begin{lem} \label{circle shell}
For any $x\in \CX\setminus\CX_\reg$ and any open neighbourhood $U\seq \CX $ of $x$, there exists an open neighbourhood $U_x \seq \CX$ of $x$ and an open subset $W$ of $\CX_\reg$ with compact closure $\overline{W}\seq \CX_\reg$ such that,  any $y\in U$ is lying on a curve $C\seq \CX_{\pi(y)}$, and $y$ has a connected open neighbourhood $V$ in $C\cap U$ with $\partial V \seq W\cap U $. 
\end{lem}
\begin{proof}
%We may choose a small open ball $B_0\seq \CX$ centered at $x$, for example, 
Let $U_0\seq \CX$ be an open neighbourhood of $x$ with an embedding $U_0\seq \IC^{N+1}, x\mapsto 0$ such that $\pi|_{U_0}$ is compatible with the projection of $\IC^{N+1}$ to the first coordinate. And we denote by $\IC^N_t$ the fiber over $t\in\IC$ of this projection. 
Let $B_{r_0 }(x )$ be the restriction of the standard open ball of $\IC^{N+1}$ with radius $r_0 >0$ on $\CX$. Without loss of generality, we assume that $B_{r_0} (x) \cap \CX \subset U $. 

Note that $\dim(\CX_t)=n$ and $\dim(\CX\setminus \CX_\reg)\le n-1$, we can choose a $(N-n+1)$-dimensional linear subspace $L_0\seq \IC^{N}_0$ such that $L_0\cap \CX_0$ is pure of dimension one near $x$ and $L_0\cap (\CX\setminus\CX_\reg)$ is a discrete set of points (both conditions are open). See also \cite[Theorem II-6.2]{dm1}. Hence there exists $r\in (\frac{r_0 }{2} ,r_0 ) $ such that $L_0\cap (\CX\setminus\CX_\reg)\cap \partial B_r(x)=\emptyset $. Now $V_0=L_0\cap\CX_0\cap B_r(x) \seq \IC^{N}_0 $ is an analytic variety of pure dimension $1$ such that $x \in V_0$ and $\partial V_0 \seq \CX_\reg $. 

Let $W$ be an open neighbourhood of $\partial V_0$ in $\CX_\reg$ with compact closure $\overline{W}\seq \CX_\reg$. 
For any $y\in B_r(x)$, we consider the $(N-n+1)$-dimensional linear subspace $L_{y,0} = L_0 + y-x \seq \IC^N_t $. Apply the above argument again, one can see that there exists a $(N-n+1)$-dimensional linear subspace $L_{y } \seq \IC^N_t $ such that $y\in L_y $, $V_y=L_y\cap \CX_{\pi(y)} \cap B_r(x)$ is an analytic variety of pure dimension $1$ and $d_H ( L_{y,0} \cap B_{1} (y ) , L_y \cap B_{1} (y ) ) \le |x-y| $, where $d_H$ is the Hausdorff distance in $\IC^{y+1} $. Note that $V_y $ converges to $V_0$ as $y\to x$ in the Hausdorff sense in $\IC^{n+1}$. Then we can find a small open neighbourhood $U_x \seq \CX$ of $x$ such that $\partial V_y$ lies in $W$. If we replace $V_y$ by a connected component containing $y$, we assume that $V_y $ is connected. Then the open subsets $W, U_x $ are as desired. 
%We may deform the linear subspace $L_0$ in $\IC^{N+1}$ to get a family of linear subspaces $L_y\seq \IC_{\pi(y)}$ with $y$ lying on a small neighbourhood of $x$ such that $V_y=L_y\cap \CX_{\pi(y)} \cap B_x(r)$ is a connected analytic curve with finitely many singularities, whose boundary in $C_y=L_y\cap \CX_{\pi(y)} $ lies in $W$. We emphasize that the first two conditions defining $L_y$ are open, but the third condition depends on $W$. 
\end{proof}

\begin{prop}
For any compact subset $K\seq C$, we have 
$$\mathop{\limsup}_{m\to \infty} \,
\mathop{\sup}_{x\in \pi^{-1}(K)}\frac{\log(\rho_m)}{m}\le 0. $$
\end{prop}

\begin{proof}
Fix $0<\delta<1$. For any $x\in \CX\setminus\CX_\reg$, we choose an open neighbourhood $U$ of $x$ in $\CX$ and holomorphic frames $e_j$ of $F_j$ on $U$ such that $ 1-\delta \le \lVert e_j \rVert_{h_j}^2 \leq 1+\delta $. Then we have open subsets $U_x ,W\seq \CX$ as in Lemma \ref{circle shell}. 
%Shrinking $U$ if necessary, we may assume that $\CL|_U$ is trivial, with generator $e\in H^0(U,\CL|_U)$. 
For any $y\in U_x ,$ $ t=\pi(y)$, we may choose $s\in H^0(\CX_t, \CL_t)$ such that $\rho_m(y)=\lVert s(y) \rVert_{\prod_{j=1}^k h_{j}^{p_{j,m}} }^2$. We can write $s=f\prod_{j=1}^k e_{j}^{p_{j,m}}$ for some $f\in \CO_{\CX_t}(U\cap\CX_t)$, then $\rho_m(y)=|f(y)|^2 \prod_{j=1}^k \lVert e_j \rVert_{h_j}^{2 p_{j,m}}$. By Lemma \ref{circle shell}, there is a curve $C\seq \CX_t$ passing through $y$, and a connected neighbourhood $V$ of $y$ in $C$ such that $\partial V\seq W$. Since $\lVert s(z) \rVert_{\prod_{j=1}^k h_{j}^{p_{j,m}} }^2\le \rho_m(z)$ for any $z\in W \cap \CX_t $, by maximal principle, we have 
$$\rho_m(y)=|f|^2 \prod_{j=1}^k \lVert e_j \rVert_{h_j}^{2 p_{j,m}}
\le  (1-\delta)^{-C_r m} (1+\delta)^{C_r m} | f(z)|^2 \le (1-\delta)^{-C_r m} (1+\delta)^{C_r m} \rho_m(z), $$
for some $z\in W$, where $C_r >0$ is a constant depending only on $r_1 ,\cdots ,r_k$. Then we can get the estimate
\begin{eqnarray*}
    \sup_{y\in U_x } \frac{\log ( \rho_m(y) )}{m} & \le & C_r \log (1+\delta) - C_r \log (1-\delta) + \sup_{z\in \overline{W}} \frac{\log ( \rho_m(z) )}{m} \\
    & \le & 2\delta +C_r \log (1+\delta) - C_r \log (1-\delta)
\end{eqnarray*}
for sufficiently large $m $, where the second inequality follows from the argument in Proposition \ref{upper bound of the regular part}. 

Since $\pi$ is proper and $K$ is compact, we can choose a finite open cover $\{U_i\}$ of $\pi^{-1}(K)\setminus \CX_\reg$ such that any $U_i$ has the property stated in Lemma \ref{circle shell}. Then we get the global estimate 
$$\mathop{\sup}_{y\in \cup_i U_i} \frac{\rho_m(y)}{m} \le 2\delta +\log (1+\delta) - \log (1-\delta) . $$ 
We may choose an open subset $V\seq \CX_\reg\cap \pi^{-1}(K)$ with compact closure and $\CX\setminus V\seq \cup_i U_i$. Hence
$$\mathop{\sup}_{y\in \pi^{-1}(K)} \frac{\rho_m(y)}{m} \le 
\max\Big\{\mathop{\sup}_{y\in \cup_i U_i} \frac{\rho_m(y)}{m},
\mathop{\sup}_{y\in \overline{V}} \frac{\rho_m(y)}{m} \Big\}
\le 2\delta +\log (1+\delta) - \log (1-\delta) , $$
for sufficiently large $m $. We conclude by taking $\limsup_{m\to \infty}$ and $\delta\to 0$. 
\end{proof}

{\it Step 3}. Lower bound on the regular part. Now we will give an estimate of the pointwise lower bound of $\frac{\log(\rho_m)}{m}$ on compact subsets of $\CX_\reg $. We need to use H\"ormander's $L^2 $ estimate and the desingularization of analytic varieties in this step.

We first state the H\"omander's $L^2$ estimate without proof. The proof can be found in \cite{dm1, lh1}. Note that the $L^2$ estimate holds also on weakly pseudoconvex manifold (the metric is not assumed to be complete).
\begin{prop}
\label{prophormanderl2m}
Let $(M,\omega )$ be an $n$-dimensional complete K\"ahler manifold. Let $(L ,h)$ be a Hermitian holomorphic line bundle, and $\psi $ be a function on $M$, which can be approximated by a decreasing sequence of smooth functions $\left\lbrace \psi_i \right\rbrace_{i=1}^{\infty} $. Suppose that $$\sqrt{-1}\partial\bar{\partial} \psi_i + \Ric(\omega )+\Ric(h)  \geq c\omega $$
for some positive constant $c >0 $. Then for any $L $-valued $(0,q)$-form $\zeta\in L^{2}$ on $M$ with $\bar{\partial} \zeta =0$ and $\int_{M} ||\zeta ||^{2} e^{-\psi} \omega^n $ finite, there exists an $L $-valued $(0,q-1)$-form $u\in L^{2}$ such that $ \bar{\partial} u =\zeta$ and $$\int_{M} \Vert u\Vert^{2} e^{-\psi } \omega^n \leq \int_{M} c^{-1} \Vert \zeta \Vert^{2} e^{-\psi} \omega^n ,$$
where $||\cdot ||$ denotes the norms associated with $h$ and $\omega $, and $q=1,\cdots ,n$.
\end{prop}

We will use the following version of resolution of singularities of analytic varieties. See for example \cite[Theorem 0.2]{km98} or \cite{hir64, bm97}.

%\begin{thm}  Let $X$ be a reduced irreducible analytic variety, and $I\seq X$ be a coherent sheaf of ideals defining a closed subvariety $Z\seq X$. Then there are a smooth variety $Y$ and a projective morphism $f:Y\to X$ such that {\rm (1)} $f$ is an isomorphism over $X\setminus (Sing(X)\cup Supp(Z))$, \\ \indent {\rm (2)} $f^*I\seq \CO_Y$ is an invertible sheaf $\CO_Y(-D)$ and \\ \indent {\rm (3)} $Ex(f)\cup D$ is a simple normal crossing divisor. \end{thm}

\begin{thm}
\label{thmdesingularization}
Let $\CX$ be a reduced irreducible analytic variety, and $Z\subset \CX$ be a closed subvariety. Then there are a smooth analytic variety $\tilde{\CX}$ and a projective holomorphic morphism $\sigma : \tilde{\CX} \to \CX $ such that $Ex(\sigma)\cup\sigma^{-1} (Z) $ is a locally finite normal-crossing divisor of $\tilde{\CX}$.

Moreover, $\sigma$ is the composition of a locally finite sequence of blowing-ups on $\CX$, and the centers are smooth subvarieties of $\CX_{\sing} \cup Z $.
\end{thm}
The following construction of metrics on $\tilde{\CX}$ is essentially due to Mo\u i\u sezon \cite[Lemma 1]{moi67} and Coman-Ma-Marinescu \cite[Lemma 2.2]{cmm17}.

\begin{lem}
\label{lemmoishezonconstruction}
Let $\CX$ be a reduced irreducible analytic variety with smooth Hermitian metric $\omega $, and $\sigma : \tilde{\CX} \to \CX $ be a resolution of singularities as in Theorem \ref{thmdesingularization}, with exceptional divisor $E$. Then for any compact subset $K\subset \tilde{\CX} $, there exist a smooth Hermitian metric $\theta $ on $\CO_{\tilde{\CX}} (-E) |_{K} $ and a constant $C_K >0$ such that $C_K \sigma^* \omega +c_1 (\theta ) >0 $ on $K$.
\end{lem}
\begin{proof}
We only give an outline of the proof of this lemma here. A detailed proof can be found in \cite[Lemma 1]{moi67} or \cite[Lemma 2.2]{cmm17}.

Since $\sigma$ is projective, 
%the composition of a locally finite sequence of blowing-up on $\CX$. 
we can find an open covering $\{ U_i \}_{i\in I} $ of $\sigma (K)$, such that $U_i $ is a analytic subvariety of an open subset $W_i$ of $\mathbb{C}^{N_i}$, and $\sigma^{-1} (U_i) $ can be embedded in $W_i \times \mathbb{CP}^{N'_i} $ such that $\sigma $ is the restriction of the projection $W_i \times \mathbb{CP}^{N'_i} \to W_i $ on $\sigma^{-1} (U_i) $, and $\CO_{\tilde{\CX}} (-E) = \CO_{\mathbb{CP}^{N'_i}} (1) $ on $\sigma^{-1} (U_i) $.

By a smooth partition of unity on $X$, we can construct the smooth Hermitian metric $\theta $ on the line bundle $\CO_{\tilde{\CX}} (-E) |_{K} $ that we need.
\end{proof}

We are ready to give an estimate of the pointwise lower bound of $\frac{\log(\rho_m)}{m}$ on compact subsets of $\CX_\reg $. Our argument is similar to the one in Coman-Ma-Marinescu \cite[Theorem 1.1]{cmm17}, using the peak section method for the lower bound estimation.

\begin{prop}
\label{proppart3lowerestimate}
Let $(\CX , \omega,\CL_m ,h^{\CL_m } )$ be a system as in the beginning of this section. Then for any compact subset $K\subset \CX_\reg$, we have 
$$\mathop{\limsup}_{m\to \infty} \,
\mathop{\sup}_{x\in K}\frac{|\log(\rho_m)|}{m}= 0. $$
\end{prop}

\begin{proof}
Fix $\delta >0$ and $t\in C$. For any $x\in\CX_{t,\red }$, we can find an open neighbourhood $W_x $ of $x$ in $\CX$ and a holomorphic embedding $\tau_x : W_x \to U_x \subset \mathbb{C}^{N_x} $ such that $F_i $ are trivial line bundle on $W_x $, and $\tau_x (W_x ) $ is a closed subvariety of $U_x $, where $U_x $ is an open subset of $\mathbb{C}^{N_x} $. Without loss of generality, we assume that $U_x$ is the unit ball in $\mathbb{C}^{N_x} $. Choosing local frames $e_{i,x} \in H^0 (W_x ,F_i ) $, we may write $\psi_{i,x } = -\log h_i (e_{i,x} ,e_{i,x} ) $. By definition, $\psi_{i,x }$ are plurisubharmonic functions on $W_x$. 

Then $\psi_{i,x }$ can be extended to a continuous quasi-plurisubharmonic function $\tilde{\psi}_{i,x} $ on $U_x$ such that $\iddbar \tilde{\psi}_{i,x} \geq -\delta \omega_{ \rm Euc } $, $i=1,\cdots ,k$, where $\omega_{\rm Euc}$ is the Euclidean metric on $\mathbb{C}^{N_x}$. Moreover, we assume that $\iddbar \tilde{\psi}_{1,x} \geq \epsilon' \omega_{\rm Euc} $ for some constant $\epsilon' >0$ independent of $\delta$. See also \cite{cgz13}. By the standard regularization as in \cite{bk07, dp04}, we can construct smooth Hermitian metrics $\tilde{h}_i $ on $F_i |_{\pi^{-1} (\D_t ) } $, where $\D_t $ is an open neighbourhood of $t\in C$, such that $ 1-\delta
\leq \frac{\tilde{h}_i}{h_i} \leq 1+\delta $, $c_1 ( \tilde{h}_{i } ) \geq -\delta \tilde{\omega} $ on $\pi^{-1} (\D_t ) $, $i=1,\cdots ,k$, and $c_1 ( \tilde{h}_{1 } ) \geq \epsilon' \tilde{\omega} $ on $\pi^{-1} (\D_t ) $, where $\tilde{\omega} $ is a smooth Hermitian metric on $\pi^{-1} (\D_t ) $, and $\epsilon' >0$ is a constant independent of $\delta$. Since $c_1 ( \tilde{h}_{1 } ) >0 $, we can choose $\tilde{\omega} = c_1 ( \tilde{h}_{1 } ) $, hence $\tilde{\omega} $ is a K\"ahler metric on $\pi^{-1} ( \D_t ) $. Without loss of generality, we assume that $\epsilon' \geq 100\delta \sum_{i=1}^k r_i $. 

Now we consider the resolution of singularities $\sigma : \tilde{\CX} \to \CX $. By Lemma \ref{lemmoishezonconstruction}, there exists an integer $d\in\mathbb{N}$, and a Hermitian metric $\theta$ on $\mathcal{O}_{\tilde{\CX}} |_{\pi^{-1} (\D_t )} $ such that for any $m\in\mathbb{N}$, $ c_1 (\tilde{h}_{d,m}) >0 $ on $\mathcal{O}_{\tilde{\CX}} |_{\pi^{-1} (\D_t )} $, where $\tilde{h}_{d,m} = \prod_{j=1}^k  \sigma^* \tilde{h}^{p_{j,m+d} -p_{j,m} }_{j } \otimes \theta $. Note that we can replace the open set $\D_t$ by a precompact subset. Let $s_E \in H^0 (\tilde{\CX} ,\mathcal{O}_{\tilde{\CX}} ( E ) ) $ such that $ [s_E = 0] =E $.

Let $x\in K \cap \pi^{-1} ( \D_t ) \subset \CX_\reg \cap \pi^{-1} ( \D_t ) $. Then we can find open neighbourhoods $W_x \Subset W'_x \Subset W''_x $ of $x$ and compactly supported smooth $\CL $-valued functions $\xi_m$ on $W''_x $ such that $\xi_m$ are holomorphic on $W'_x$, $ \left\Vert \xi_m \right\Vert +\left\Vert \bar{\partial} \xi_m \right\Vert_{\tilde{\omega} } \leq C_x (1+\delta )^m $, and $ \left\Vert \xi_m \right\Vert \geq C^{-1}_x (1-\delta )^m $ on $W_x$, where $C_x >0 $ is a constant independent of $m$. Choose $\psi_y (z) = 10n^2 \eta_x \log \left| \tau_x (z) - \tau_x (y) \right| $, where $\eta_x$ is a cut off function such that $\eta_x =1 $ on $ W'_x $ and $\eta_x =0 $ outside $W''_x $. Note that $\sigma^{-1} ( \pi^{-1} (\D_t ) )$ is a pseudoconvex manifold. Then for sufficiently large $m$, we can apply H\"ormander's $L^2$ estimate to find a holomorphic section $s'_{x,y} \in H^0 \left( \sigma^{-1} ( \pi^{-1} (\D_t ) ) , \sigma^* L_m \otimes \mathcal{O}_{\tilde{\CX}} (-E)^{\lfloor \sqrt{m} \rfloor } \right) $ such that 
$$\Vert s'_{x,y} (y) \Vert_{ \prod_{j=1}^k  \sigma^* \tilde{h}^{ p_{j,m -d\lfloor \sqrt{m} \rfloor } }_{j } \otimes \prod_{j=1}^{\lfloor \sqrt{m} \rfloor } \tilde{h}_{d,m-dj} } = \left\Vert \xi_m (y) \right\Vert_{\prod_{j=1}^k  \sigma^* \tilde{h}^{ p_{j,m} }_{j }} \left\Vert s_E (y) \right\Vert_{\theta^{-\lfloor \sqrt{m} \rfloor}} ,$$
and $ \int_{\sigma^{-1} ( \pi^{-1} (\D_t ) ) } \Vert s'_{x,y} (y) \Vert^2 \leq C_x (1+\delta )^m $ where $\lfloor \cdot \rfloor $ is the greatest integer function. By definition, we can find constants $0<a<b$ such that $\left\Vert s_E \right\Vert_{\theta^{-1}} \leq b $ on $\sigma^{-1} ( \pi^{-1} (\D_t ) ) $ and $ \left\Vert s_E \right\Vert_{\theta^{-1}} \geq a $ on $K$. Let $s'_{x,y} = s''_{x,y} \otimes s_{E}^{\lfloor \sqrt{m} \rfloor} $ on $\sigma^{-1} ( \pi^{-1} (\D_t ) )  \setminus E$. Since $\sigma |_{\tilde{\CX}\setminus E} $ is biholomorphic morphism and $\CX$ is normal, $s''_{x,y}$ can be extended to a holomorphic section on $ \pi^{-1} (\D_t ) $, and $ \left\Vert s''_{x,y} (y) \right\Vert_{h^{\CL_m}}^2 \geq C_x^{-1} (1-\delta)^m $. As the arguments in Step 1 and Step 2, one can easy to see that $\left\Vert s''_{x,y} \right\Vert_{h^{\CL_m}}^2 \leq C_x (1+\delta)^m $ on $\pi^{-1} (\D_t ) $.

Combining the above results, for any constant $\delta >0$ and $x\in K$, we have an open neighbourhood $W_x$ of $x$ such that
$$\mathop{\limsup}_{m\to \infty} \,
\mathop{\sup}_{y\in K \cap W_x }\frac{|\log(\rho_m)|}{m}\leq \delta . $$

By the compactness of $K$, we conclude by letting $\delta \to 0 $.
\end{proof}

{\it Step 4}. Integration bound around the singular part. By giving a uniform estimate for the integral of $ \frac{ |\log (\rho_m ) | }{m} $ around the singular part, we will prove Theorem \ref{uniformly convergence section 4} in this step. We give the following estimate.

\begin{lem}
\label{lemmalogfintegral}
Let $\left( \CL ,h \right) $ be a line bundle with a continuous Hermitian metric on $\CX $, and $s\in H^0 (\CX ,\CL) $. Assume that $ \dim (\{s=0\} \cap \CX_{t_0} ) \leq n-1 $ for some $t_0 \in C $. Then for any  $\delta >0$, there exists an open neighbourhood $U_\delta $ of $ \CX_{t_0} \cap \{s=0\} $ in $\CX$ such that 
$$ \int_{\CX_t \cap U_{\delta} } \left| \log \Vert s \Vert \right| \omega_t^n \leq \delta ,\;\; \forall t\in C .$$
\end{lem}

\begin{proof}
Since this question is local on $C$, we assume that $C=\D$ and $t_0 =0 $. Let $x\in \CX_0 \cap \{s=0\} $. Then we can find open neighbourhoods $W'_x \Subset W_x $ of $x$ in $\CX$ and a holomorphic embedding $\tau_x : W_x \to U_x \subset \mathbb{C}^{N_x} $ such that $\CL $ is trivial line bundle on $W_x $, and $\tau_x (W_x ) $ is a closed subvariety of $U_x $, where $U_x $ is an open subset of $\mathbb{C}^{N_x} $. Let $\tilde{\omega } = \tau_x^* \omega_{\rm Euc} $, and $e_x \in H^0 (W_x ,\CL ) $ be a local frame. Write $s=fe_x$ on $W_x$. By the compactness of $ \CX_0 \cap \{s=0\} $, it suffices to show that for any given $\delta >0$, there exists an open neighbourhood $U_\delta $ of $ \CX_0 \cap \{s=0\} \cap W'_x $ such that 
$$ \int_{\CX_t \cap U_{\delta} } | \log |f| | \tilde{\omega }_t^n \leq \delta ,\;\; \forall t\in C .$$

Let $\sigma : \tilde{\CX} \to \CX $ be a resolution of singularities such that $\sigma^{-1} (\CX_0 \cup \{s=0\} ) $ is a normal crossing divisor of $\tilde{\CX} $. Then for any $y\in \sigma^{-1} ( \CX_0 \cap \{s=0\} \cap \bar{W}'_x ) $, there exist an open neighbourhood $\CU_{y} $ of $y\in\tilde{\CX}$ and holomorphic coordinates $z_1 ,\cdots ,z_{n+1} $ on $\CU_y$ such that $\sigma^* \CX_0 = \sum_{i=1}^{n_1 +n_2 } l_i \{z_i=0\} $ and $\sigma^* (\{f=0\}) = \sum_{i=n_1 }^{n_1 +n_2 +n_3 } l'_i \{z_i=0\} $. Without loss of generality, We may assume that the coordinates $z= (z_1,\cdots , z_i ,\cdots ,z_{n+1} ) : \CU_y \to \mathbb{C}^{n+1} $ gives a biholomorphic map between $\CU_y $ and $\D^{n+1}$, $\pi\circ \sigma = \prod_{i=1}^{n_1 +n_2} z_i^{l_i} $ and $f\circ \sigma = \prod_{i=n_1 }^{n_1 +n_2 +n_3 } z_i^{l'_i} $ on $\CU_y$. Then $\log |f| = \sum_{i=n_1 }^{n_1 +n_2 +n_3 } l'_i \log | z_i |  $.

Let $\alpha = \sigma^* \tilde{\omega }^n $. Since $ \dim (\{s=0\} \cap \CX_{t_0} ) \leq n-1 $, for $i=n_1 ,\cdots , n_1 +n_2 $, there exists a non-vanishing vector field $X_i$ on $\CU_y \cap \{z_i=0\} $ such that $d\sigma (X_i) =0 $. It follows that for $i=n_1 ,\cdots , n_1 +n_2 $, there exist a constant $C_{\alpha} >0 $ and an analytic real function $\phi_i $ on $\CU_y$ such that $\phi_i =0 $ on $\CU_y \cap \{z_i=0\} $, and 
$$\alpha \leq C_{\alpha} \sqrt{-1} dz_{i} \wedge d\bar{z}_i \wedge ( \sum_{j\neq i} \sqrt{-1} dz_j \wedge d\bar{z}_j )^{n-1} + \phi_i \cdot ( \sum_{j\neq i} \sqrt{-1} dz_j \wedge d\bar{z}_j )^{n} .$$
Shrinking the domain $\CU_y$ if necessary, we assume that $ |\phi_i |\leq C_{\alpha} |z_i| $. Fix $\epsilon >0$. Then for any $i=n_1 ,\cdots , n_1 +n_2 $, we have 
\begin{eqnarray*}
& & \int_{\sigma^{-1} (\CX_t ) \cap \{ |z_i| \leq \epsilon \} } \left| \log |z_i | \right| \alpha \\
& \leq & \int_{\sigma^{-1} (\CX_t ) \cap \{ |z_i| \leq \epsilon \} } \left| \log |z_i | \right| C_{\alpha} \sqrt{-1} dz_{i} \wedge d\bar{z}_i \wedge ( \sum_{j\neq i} \sqrt{-1} dz_j \wedge d\bar{z}_j )^{n-1} \\
& & + \int_{\sigma^{-1} (\CX_t ) \cap \{ |z_i| \leq \epsilon \} } \left| \log |z_i | \right|  \phi_i \cdot ( \sum_{j\neq i} \sqrt{-1} dz_j \wedge d\bar{z}_j )^{n} \\
& \leq & C_{\alpha } \sum_{j\neq i} \int_{ ([z_j =0])  \cap \{ |z_i| \leq \epsilon \} } |\log |z_i| | \omega^n_{\rm Euc} + C_{\alpha } \int_{ ([z_i =0]) } |z_i| |\log |z_i| | \omega^n_{\rm Euc} \\
& \leq & C_{\alpha } \int_{w\in\D , |w|\leq \epsilon  } |\log |w| | \omega_{\rm Euc} + C_{\alpha } \int_{w\in\D } |w| |\log |w| | \omega_{\rm Euc} \leq C_{\alpha } \epsilon | \log ( \epsilon ) |,
\end{eqnarray*}
where $C_{\alpha } >0 $ is a constant independent of $\epsilon$. With the similar argument, the inequality 
$$ \sum_{i=n_1}^{n_1 +n_2 +n_3 } \int_{\sigma^{-1} (\CX_t ) \cap \{ |z_i| \leq \epsilon \} } \left| \log |z_i | \right| \alpha \leq C_{\alpha } \epsilon | \log ( \epsilon ) | $$
holds for sufficiently small $\epsilon>0$

Let $\CU_{y,\epsilon} = \cup_{i=n_1}^{n_1 +n_2 +n_3} \{ |z_i| < \epsilon \}  $, which is an open neighbourhood of $\CU_y \cap \sigma^{-1} (\{s=0\}) \cap \sigma^* (\CX_0 ) $. Now the lemma follows from the compactness of $W'_x$.
\end{proof}

We are ready to prove the main theorem of this section. 

\begin{proof}[Proof of Theorem \ref{uniformly convergence section 4}]
It suffices to show that for any constant $\delta>0$ and conpact subset $K\subset C$, there exists a compact subset $\CK$ of $\CX_\reg$ such that $$\limsup_{m\to\infty} \sup_{t\in K} \int_{\CX_t \setminus \CK } \frac{1}{m}
|\log(\rho_m)|\omega_t^n \leq \delta .$$

Since this question is local on $C$, we assume that $C=\D$.

Let $d>\sum_{i=1}^{k} \frac{100}{r_i} $ be an integer, and let $A_j = \prod_{i=1}^{k} [ jr_i - \frac{10j}{d} , jr_i + \frac{10j}{d} ] \cap \mathbb{Z}^k $. For any $\xi\in A_j$, write $\xi = ( \zeta_1 ,\cdots ,\zeta_k ) $, $\CL_{\xi} = \otimes_{i=1}^k F_i^{\zeta_i} $. We define $\rho_{\xi}$ to be the fiberwise Bergman kernel of the bundle $\CL_{\xi}$. As in step 3, one can see that for sufficiently large $d$, there exists a holomorphic section $s_{\xi}$ of $\CX_{\xi}$ around $\CX_0$ such that $ \dim (\{s_{\xi}=0\} \cap \CX_{0} ) \leq n-1 $, $\forall \xi \in \cup_{j=d}^{2d} A_j $. Then Lemma \ref{lemmalogfintegral} implies that for any $ \xi \in \cup_{j=d}^{2d} A_j $, there exist an open neighbourhood $U_{\xi} $ of $ \CX_{0} \cap \{s_{\xi}=0\} $ in $\CX$ and an open neighbourhood $\D_\xi $ of $0\in\D$, such that $ \pi^{-1} (\D_\xi ) \cap \{s_\xi = 0 \} \subset U_{\xi} $, and
$$ \int_{\CX_t \cap U_{\xi} } \left| \log \Vert s_{\xi} \Vert \right| \omega_t^n \leq \delta ,\;\; \forall t\in C .$$
Let $\CU_d = \cap_{\xi \in \cup_{j=d}^{2d} A_j } U_{\xi} $ and $\D_d = \cap_{\xi \in \cup_{j=d}^{2d} A_j } \D_{\xi} $. By induction, we see that for any $j\geq 2d$ and $\xi \in A_j $, there are $\xi_1 ,\cdots ,\xi_l \in \cup_{q=d}^{2d} A_q $ such that $\sum_{p=1}^l \xi_p =\xi $. It follows that
$$\limsup_{m\to\infty} \sup_{t\in C} \int_{\CX_t \cap \CU_d } \frac{1}{m}
|\log(\rho_m)|\omega_t^n \leq \delta .$$
With a similar argument, one may show that there exists a constant $C_d >0 $ such that
$$\limsup_{m\to\infty} \sup_{x\in \pi^{-1} (\D_d ) \setminus \CU_d } \frac{1}{m}
|\log(\rho_m (x) ) | \leq C_d .$$
Combining the previous results, we see that there exists a compact subset $\CK$ of $\CX_\reg$ such that
$$\limsup_{m\to\infty} \sup_{t\in \D_d} \int_{\CX_t \setminus \CK } \frac{1}{m}
|\log(\rho_m)|\omega_t^n \leq \delta .$$
By the compactness of $K$, we can replace $\D_d $ in the above inequality by $K$. The proof is complete.
\end{proof}

\section{Applications}
\label{section_Applications}

\subsection{Relatively ample line bundles}

Let $\CX, C$ be complex varieties with a flat proper morphism $\pi:\CX\to C$. Let $\CL$ be a line bundle on $\CX$ which is relatively ample with respect to $\pi$. Recalling that we say that $\CL$ is relatively ample with respect to $\pi$ if for some $r\in \IN$, there exist sections $s_0,\cdots,s_N \in H^0(\CX, r\CL)$ inducing a closed embedding $\CX \to \IP^N\times C$, which is compatible with $\pi$, and $r\CL$ is called relatively very ample with respect to $\pi$. 

\begin{thm}\label{relamplethm}
Assume that $C$ is compact in Zariski topology. Then there exists a constant $m_0\in \IN$ such that, for any $m\ge m_0$, the function $h^0(\CX_t, mr\CL_t)$ is constant for all closed points $t\in C$. 
\end{thm}

This result may be well known to experts. For the convenience of the reader, we state a proof here. 

\begin{defi} \rm
Let $X$ be a noetherian scheme, and $L$ be a fixed globally generated ample line bundle on $X$. A coherent sheaf $\CE$ on $X$ is called {\it regular} if $H^i(X, \CE)=0$ for all $i>0$. For any $m\in \IN$, $\CE$ is called $m$-{\it regular} if $H^i(X,\CE(m-i))=0$ for all $i>0$, where $\CE(j)=\CE\otimes L^{j}$. 
\end{defi}

\begin{lem} \label{regular sheaf}
Let $X\seq \IP^N_A$ be a projective scheme over a noetherian ring $A$, and $L$ be a globally generated ample line bundle on $X$. For any coherent sheaf $\CE$ on $X$, if $\CE$ is $m_0$-regular for some $m_0\in \IN$, then $\CE$ is $m$-regular for all $m\ge m_0$. 
\end{lem}

\renewcommand{\proofname}{Proof of Theorem {\ref{relamplethm}}}
\begin{proof}
Here we only prove the case where $\pi $ is a flat morphism between algebraic varieties over $\IC$. The proof for the analytic case is almost the same as the proof for the algebraic case.

Since $\pi$ is flat and $\CL$ is a line bundle, by \cite[Theorem 9.9]{har77}, we see that the Hilbert polynomial
$$\chi(\CX_t, mr\CL_t):=\sum_i(-1)^{i} h^i(\CX_t, mr\CL_t)$$
is independent of $t\in C$. It suffices to show that $h^i(\CX_t, mr\CL_t)=0$ for all $i>0, m\ge m_0$ and $t\in C$, where $m_0$ is a constant depending on $\CL$.

For any $t_1 \in C$, the restriction $r\CL_{t_1}$ is ample on $\CX_{t_1}$. By \cite[Theorem 5.2]{har77}, there exists $m_1'\in \IN$ such that $h^i(\CX_{t_1}, mr\CL_{t_1})$ vanishes for all $m\ge m_1'$. We may assume that $\CX_{t_1}\seq \IP^N$, then $H^i(\CX_{t_1}, \CE)=0$ for all $i>N$ and coherent sheaf $\CE$ on $\CX_{t_1}$. Let $m_1=m_1'+N$. By the upper-semicontinuity of the function $h^i(\CX_t, mr\CL_t)$ on $C$, for any $1\le i\le N$, there exists a Zariski open neighbourhood $U_1^i$ of $t_1$ such that $h^i(\CX_t, (m_1-i)r\CL_t)=0$ for all $t\in U_1^i$. We denote by $U_1:=\cap_{i=1}^N U_1^i$, then the structure sheaf $\CO_{\CX_t}$ is $m_1$-regular (with respect to $r\CL_t$) for any $t \in U_1$. Hence by Lemma \ref{regular sheaf}, we see that $mr\CL_t$ is regular for any $m\ge m_1=m_1'+N$ and $t\in U_1$. 

Note that $C\setminus U_1\seq C$ is a Zariski closed subset. For any irreducible component $D_j$ of $C\setminus U_1$, with the same procedure, we may find $m_{2j}\in \IN$ and open subset $U_{2j} \seq C$ with $U_{2j}\cap D_j\ne \varnothing$, and $mr\CL_t$ is regular for any $t\in U_{2j}$ and $m\ge m_{2j}$. 

Now let $U_2$ be the union of all the $U_{2j}$ (where $D_j$ runs over all the irreducible components of $C\setminus U_1$) and $m_2$ be the maximum of all the $m_{2j}$. We see that $mr\CL_t$ is regular for any $t\in U_1\cup U_2$ and $m\ge \max\{m_1, m_2\}$. Note that $C\setminus (U_1\cup U_2)\seq C$ is a Zariski closed subset of codimension at least two. 
Hence by noetherian induction, we can find $m_0$ such that $mr\CL_t$ is regular for any $t\in B$ and $m\ge m_0$. 
\end{proof}
\renewcommand{\proofname}{Proof}

Combining Theorem \ref{maintheorem}, Theorem \ref{uniformly convergencechap1} and Theorem \ref{relamplethm}, one can obtain the following result.

\begin{prop}
Let $(\CX ,\CL ,\omega ,h,C,\pi )$ be as in Theorem \ref{maintheorem}. Assume that $\CL $ is relatively ample  with respect to $\pi$, and $C$ is compact in Zariski topology. Then there exists a constant $m_0 >0$ such that for any $m\geq m_0$ the fibrewise Bergman kernel, $\rho_{\CL^m_{\pi (x)} } (x) $ is continuous on $\CX$. Moreover, $ \int_{\CX_t}\frac{1}{m}
|\log(\rho_m)|\omega_t^n$ converges to $0$ uniformly on compact subsets of $C$ as $m\to\infty$.
\end{prop}

\subsection{Test configurations} \label{tc}
Let $(X, L)$ be a polarized projective variety of dimension $n$. We recall the definition of {\it test configurations} (or {\it special degenerations}) introduced by \cite{tia97, don02} in studying K-stability. 

\begin{defi} \rm
A {\it test configuration} $(\CX, \CL)$ of a polarized projective variety $(X, L)$ consists of the following data

(1) A normal variety $\CX$ and a projective morphism $\pi:\CX \to \IC$;

(2) A $\pi$-semiample $\IQ$-line bundle $\CL$ on $\CX$;

(3) A $\IC^*$-action on $(\CX, \CL)$ such that $\pi$ is $\IC^*$-equivariant and induces a $\IC^*$-equivariant isomorphism 
$$(\CX,\CL)|_{\pi^{-1}(\IC^*)}\cong (X,L)\times \IC^*.$$ 
The test configuration $(\CX,\CL)$ is called ample if $\CL$ is $\pi$-ample. 
\end{defi}

For any test configuration $(\CX, \CL)$, we will see that $(\CX, r\CL)$ is extendable for any integer $r\in \IN$ such that $r\CL$ is a line bundle. Hence our results apply.

We denote by $R=R(X,L)=\oplus_{m\in \IZ_{\ge0}} R_m$ where $R_m=H^0(X, mL)$. 

\begin{defi} \rm
A ($\IZ$-){\it filtration} $\CF$ on $R$ is a sequence of subspaces $\CF^\lam R_m \subseteq R_m$ for each $\lam \in \IZ$ and $m \in \IN$ such that

(F1) {\it Descending.} $\CF^\lam R_m \supseteq \CF^{\lam'}R_m $ for  $\lam \le \lam'$; 

(F2) {\it Linearly Bounded.} There exists $C>0$ such that $\CF^\lam R_m = R_m$ for $\lam < -Cm$ and $\CF^\lam R_m = 0$ for $\lam > Cm$; 

(F3) {\it Multiplicable.} $\CF^\lam R_m \cdot \CF^{\lam'}R_{m'} \subseteq \CF^{\lam+\lam'}R_{m+m'}$. 
\end{defi}

For any test configuration $(\CX, \CL)$ of $(X,L)$, we may define a filtration on $R=R(X,rL)$, where $r\in \IN$ is a integer such that $r\CL$ is Cartier (a line bundle). Then we have the restriction map  
$$\iota: H^0(\CX, mr\CL)\to H^0(\CX_1, mr\CL_1)=H^0(X,mrL), m\in \IN. $$
The $\IC^*$-action on $(\CX, r\CL)$ induces a weight decomposition $H^0(\CX,mr\CL)=\oplus_{\lam \in \IZ} H^0(\CX, mr\CL)_\lam$. We define a filtration on $H^0(X, mrL)$ by 
$\CF^\lam H^0(X, mrL) = \iota(H^0(\CX, mr\CL)_\lam)$. In the other word
$$\CF^\lam H^0(X, mrL) = \{s\in H^0(X, mrL): t^{-\lam}\bar{s} \in H^0(\CX, mr\CL)\}, $$
where $\bar{s}$ is the $\IC^*$-invariant section on $\CX\setminus \CX_0$ induced by $s$.

Conversely, we have the following Rees construction. 

\begin{defi} \rm
The {\it Rees algebra} $\Rees_\CF(R)$ of a filtration $\CF$ on $R$ is defined by 
$$\Rees_\CF(R):= \oplus_{m\in\IN, \lam\in \IZ}\,t^{-\lam}\CF^{\lam}R_m, $$
the associated graded quotient algebra is defined by 
$$\gr_\CF(R):= \oplus_{m\in\IN, \lam\in \IZ}\CF^\lam R_m/\CF^{\lam+1}R_m. $$
\end{defi}

A filtration $\CF$ of $R$ is called {\it finitely generated} if $\Rees_\CF(R)$ is a finitely generated $\IC[t]$-algrbra (equivalently, $\gr_\CF(R)$ is finitely generated $\IC$-algrbra). In this case, taking $\Proj$ will define a test configuration of $(X,L)$
$$(\CX_\CF, \CL_\CF):= (\Proj_{\IC[t]}\Rees_\CF(R), \CO(1)). $$
In particular, we have
$$(\CX_{\CF,0}, \CL_{\CF,0})= (\Proj_{\IC}\gr_\CF(R), \CO(1)). $$
Hence $H^0(\CX_{\CF, 0}, m\CL_{\CF, 0})=\oplus_{\lam\in\IZ}\CF^\lam H^0(X, mL)/\CF^{\lam+1} H^0(X, mL)$ for $m\in \IN$,  and $h^0(\CX_{\CF, t}, \CL_{\CF, t})$ is constant for $t\in \IC$. 

The above two constructions are inverse to each other for ample test configurations. See for example \cite[Proposition 2.15]{bhj17}. 

\begin{prop} 
For any ample test configuration $(\CX, \CL)$ of a polarized projective variety $(X,L)$, and an integer $r\in\IN$ making $r\CL$ a globally generated line bundle, the induced filtration $\CF$ on $R=R(X, rL)$ is finitely generated and 
$$(\CX, r\CL)\cong (\CX_\CF, \CL_\CF). $$
\end{prop}

Let $(\CX,\CL)$ be a merely semiample test configuration, with induced filtration $\CF$ on $R=R(X,rL)$. Then the Rees construction $(\CX_\CF, \CL_\CF)$ defines an ample model of $(\CX, r\CL)$, that is, there is a birational projective morphism $\CX\to \CX_\CF$ such that $r\CL$ is the pull-back of $\CL_\CF$. See \cite[Proposition 2.17]{bhj17} for details. Hence we have 
$$H^0(\CX_0, r\CL_0) = H^0(\CX_{\CF,0}, \CL_{\CF,0}),$$
and $h^0(\CX_t, r\CL_t)$ is constant for $t\in \IC$.

\begin{thm}[= Corollary \ref{maincoro2}]
For any (semiample) test configuration $(\CX,\CL)$ of a polarized projective variety $(X,L)$, the fiberwise Bergman kernel $\rho$ of $(\CX, r\CL, h, \omega)$ is continuous on $\CX$ for $r\in \IN$ such that $r\CL$ is a globally generated line bundle. Moreover, $ \int_{\CX_t}\frac{1}{m}
|\log(\rho_m)|\omega_t^n$ converges to $0$ uniformly on compact subsets of $C$ as $m\to\infty$.
\end{thm}


\begin{thebibliography}{BBEGZ16} 
{

\bibitem[BM97]{bm97} E. Bierstone, P. D. Milman: 
\emph{Canonical desingularization in characteristic zero by blowing up the maximum strata of a local invariant},
Invent. Math. \textbf{128} (1997), 207–302.

\bibitem[BK07]{bk07} Z. B{\l}ocki, S. Ko{\l}odziej: 
\emph{On regularization of plurisubharmonic functions on manifolds},
Proc. Amer. Math. Soc. \textbf{135} (2007), 2089–2093.

\bibitem[BHJ17]{bhj17} S. Boucksom, T. Hisamoto, M. Jonsson, 
\emph{Uniform K-stability, Duistermaat–Heckman
measures and singularities of pairs},
Ann. Inst. Fourier (Grenoble) \textbf{67} (2017), 87–139.

\bibitem[Cat97]{cat97} D.~Catlin,
\emph{The Bergman kernel and a theorem of Tian},
Analysis and geometry in several complex variables (Katata, 1997), 1–23, Trends Math., Birkhäuser Boston, Boston, MA, 1999.

\bibitem[CDS15]{cds15} X.-X. Chen, S.~Donaldson, S.~Sun,
\emph{K\"ahler-Einstein metrics on Fano manifolds, I, II, III},
J. Am. Math. Soc. \textbf{28} (2015) 183–197, 199–234, 235–278.

\bibitem[CGZ13]{cgz13} D.~Coman, V. Guedj, A. Zeriahi,
\emph{Extension of plurisubharmonic functions with growth control},
J. Reine Angew. Math. \textbf{676} (2013), 33-49.

\bibitem[CMM17]{cmm17} D.~Coman, X. N. Ma, G. Marinescu, 
\emph{Equidistribution for sequences of line bundles
on normal K\"ahler spaces},
Geom. Topol. \textbf{21} (2017), 923–962.

\bibitem[DLM06]{dlm06} X.-Z. Dai, K.-F. Liu, X.-N.~Ma,
\emph{On the asymptotic expansion of Bergman kernel},
J. Differ. Geom. \textbf{72} (2006), 1–41.

\bibitem[Dem12]{dm1} J.-P.~Demailly, 
\emph{Complex Analytic and Differential Geometry},
preprint, https://www-fourier.ujf-grenoble.fr/~demailly/manuscripts/agbook.pdf.

\bibitem[DP04]{dp04} J.-P.~Demailly, M. Paun:
\emph{Numerical characterization of the K\"ahler cone of a compact K\"ahler manifold},
Ann. of Math. \textbf{159} (2004), 1247–1274.

\bibitem[Don01]{don01} S.~Donaldson, 
\emph{Scalar curvature and projective embeddings. I},
J. Differ. Geom. \textbf{59} (2001), 479–522.

\bibitem[Don02]{don02} S.~Donaldson, 
\emph{Scalar curvature and stability of toric varieties},
J. Differ. Geom. \textbf{62} (2002), 289–349.



\bibitem[DS14]{ds14} S.~Donaldson, S.~Sun, 
\emph{Gromov-Hausdorff limits of K\"ahler manifolds and algebraic geometry},
Acta Math. \textbf{213} (2014), 63–106.

\bibitem[GS84]{gs84} H.~Grauert, R.~Remmert, 
\emph{Coherent analytic sheaves},
Grundlehren Math. Wiss., \textbf{265}. Springer, Cham, 1984.

\bibitem[Har77]{har77} R. Hartshorne, 
\emph{Algebraic geometry}, Graduate Texts in Mathematics, No. 52. Springer-Verlag, New York-Heidelberg, 1977.

\bibitem[Hir64]{hir64} H. Hironaka:
\emph{Resolution of singularities of an algebraic variety over a field of characteristic zero, I, II},
Ann. Math. \textbf{79} (1964), 109-203, 205-326.

\bibitem[H\"or73]{lh1} L.~H\"ormander:
\emph{An introduction to complex analysis in several variables},
Van Nostrand, Princeton, NJ, 1973.

\bibitem[Jia16]{jia16} W.-S. Jiang, 
\emph{Bergman kernel along the Kähler-Ricci flow and Tian's conjecture},
J. Reine Angew. Math. \textbf{717} (2016), 195–226.

\bibitem[Kin71]{kin71} J. King,
\emph{The currents defined by analytic varieties},
Acta Math. \textbf{127} (1971), 185–220.

\bibitem[Kle66]{Kle} S. L. Kleiman, 
\emph{Toward a numerical theory of ampleness},
Ann. of Math. (2) \textbf{84} (1966), 293-344.

\bibitem[KM98]{km98} J. Kollár, S. Mori, \emph{Birational Geometry of Algebraic Varieties}, Cambridge Tracts in Mathematics, Cambridge University Press, Cambridge, 1998.


\bibitem[Kod60]{kod60} K. Kodaira, {On compact complex analytic surfaces. I}, Ann. of Math. (2) \textbf{71} (1960), 111–152. 

\bibitem[Kod63]{kod63} K. Kodaira, {On compact analytic surfaces. II, III}, Ann. of Math. (2) \textbf{77} (1963), 563–626; ibid. \textbf{78} (1963), 1–40.

\bibitem[KV71]{kiv1} R. Kiehl, J.L., Verdier,
\emph{Ein einfacher Beweis des Koh\"arenzsatzes von Grauert},
Math. Ann. \textbf{195} (1971), 24–50.

\bibitem[Li12]{chili1} C. Li, 
\emph{K\"ahler-Einstein metrics and K-Stability},
Ph.D. thesis, Princeton University, 2012.

\bibitem[LL16]{ll16} C.-J. Liu, Z.-Q.~Lu, 
\emph{Abstract Bergman kernel expansion and its applications},
Trans. Amer. Math. Soc. \textbf{368} (2016), 1467–1495.

\bibitem[LS22]{ls22} G. Liu, G. Sz\'ekelyhidi, 
\emph{Gromov-Hausdorff limits of K\"ahler manifolds with Ricci curvature bounded below},
Geom. Funct. Anal. \textbf{32} (2022), 236–279.

\bibitem[Mo\u i67]{moi67} B. G. Mo\u i\u sezon, 
\emph{Reducution theorems for compact complex spaces with a sufficiently large field of meromorphic functions},
Izv. Akad. Nauk SSSR Ser. Mat. \textbf{31} (1967), 1385–1414.

\bibitem[MZ22]{mz22} X. N. Ma, W. P. Zhang, 
\emph{Superconnection and family Bergman kernels},
Math. Ann. (2022).

\bibitem[Pet16]{pet16} P.~Petersen, 
\emph{Riemannian geometry. Third edition},
Graduate Texts in Mathematics, No. 171. Springer, Cham, 2016. 

\bibitem[Rua98]{rua98} W.-D.~Ruan, 
\emph{Canonical coordinates and Bergman metrics},
Comm. Anal. Geom. \textbf{6} (1998), 589–631.

\bibitem[Sz\'e16]{sze16} G. Sz\'ekelyhidi, 
\emph{The partial $C^0$-estimate along the continuity method},
J. Amer. Math. Soc. \textbf{29} (2016), 537–560.

\bibitem[Tia90a]{tia90a} G. Tian, \emph{On a set of polarized K\"ahler metrics on algebraic manifolds},
J. Differ. Geom. \textbf{32} (1990), 99–130.
}

\bibitem[Tia90b]{tia90b} G.~Tian, 
\emph{On Calabi's conjecture for complex surfaces with positive first Chern class},
Invent. Math. \textbf{101} (1990), 101–172.

\bibitem[Tia97]{tia97} G.~Tian, 
\emph{K\"ahler-Einstein metrics with positive scalar curvature},
Invent. Math. \textbf{130}
(1997) 1–37,

\bibitem[Tia13]{tia13} G.~Tian, 
\emph{Partial $C^0$-estimate for K\"ahler-Einstein metrics},
Commun. Math. Stat. \textbf{1} (2013), 105–113. 

\bibitem[Tia15]{tia15} G. Tian, 
\emph{K-stability and Kähler-Einstein metrics},
Comm. Pure Appl. Math. \textbf{68} (2015), 1085–1156.

\bibitem[WZ21]{wz21} F. Wang, X.-H. Zhu, 
\emph{Tian's partial $C^0$-estimate implies Hamilton-Tian's conjecture},
Adv. Math. \textbf{381} (2021), 29 pp.

\bibitem[Zel98]{zel98} S.~Zelditch, 
\emph{Szeg\"o kernels and a theorem of Tian},
Internat. Math. Res. Notices. \textbf{6} (1998), 317–331.

\bibitem[Zha21a]{zha21a} K.-W.~Zhang, 
\emph{Some refinements of the partial $C^0$ estimate},
Anal. PDE \textbf{14} (2021), 2307–2326.

\bibitem[Zha21b]{zha21b} K.-W.~Zhang, 
\emph{A quantization proof of the uniform Yau-Tian-Donaldson conjecture},
 arXiv:2102.02438 (2021). 


\end{thebibliography}
\end{document}